\documentclass[a4paper,11pt]{article}
\usepackage[margin=1in]{geometry}
\usepackage[footnotesize,bf]{caption}
\usepackage{amsmath,amsfonts,amssymb,amsthm,graphicx,color,bbm,stmaryrd}

\let\leq\leqslant
\let\geq\geqslant
\let\epsilon\varepsilon

\newtheorem{theorem}{Theorem}
\newtheorem{lemma}[theorem]{Lemma}
\newtheorem{proposition}[theorem]{Proposition}
\newtheorem{corollary}[theorem]{Corollary}
\theoremstyle{definition}
\newtheorem*{remark}{Remark}

\numberwithin{equation}{section}
\newcommand{\A}{\mathcal{A}}
\newcommand{\B}{\mathcal{B}}
\newcommand{\C}{\mathcal{C}}

\newcommand{\Ga}{\Gamma}

\newcommand{\R}{\mathbb{R}}
\newcommand{\Z}{\mathbb{Z}}
\newcommand{\al}{\alpha}
\newcommand{\ep}{\varepsilon}
\newcommand{\e}{{\rm e}}
\newcommand{\g}{\gamma}
\newcommand{\ie}{\emph{i.e.}}
\newcommand{\ph}{\phi_{p,q}}

\title{\bfseries The self-dual point of the two-dimensional \\
                 random-cluster model is critical for $q\geq 1$}
\author{V\!. Beffara \and H. Duminil-Copin}

\begin{document}

\maketitle

\begin{abstract}
  We prove a long-standing conjecture on random-cluster models, namely 
  that the critical point for such models with parameter $q\geq1$ on the 
  square lattice is equal to the self-dual point $p_{sd}(q) = \sqrt{q} / 
  (1+\sqrt{q})$. This gives a proof that the critical temperature of the 
  $q$-state Potts model is equal to $\log (1+\sqrt q)$ for all $q\geq 
  2$. We further prove that the transition is sharp, meaning that there 
  is exponential decay of correlations in the sub-critical phase. The 
  techniques of this paper are rigorous and valid for all $q\geq 1$, in 
  contrast to earlier methods valid only for certain given $q$. The 
  proof extends to the triangular and the hexagonal lattices as well.
\end{abstract}

\section*{Introduction}

Since random-cluster models were introduced by Fortuin and Kasteleyn in 
1969~\cite{FortuinKasteleyn}, they have become an important tool in the 
study of phase transitions. The spin correlations of Potts models are 
rephrased as cluster connectivity properties of their random-cluster 
representations. This allows the use of geometric techniques, thus 
leading to several important applications. Nevertheless, only a few 
aspects of the random-cluster models are understood in full generality.

The \emph{random-cluster model} on a finite connected graph is a model 
on the edges of the graph, each one being either closed or open. The 
probability of a configuration is proportional to $$p^{\# \;\text{open 
edges}} (1-p)^{\# \; \text{closed edges}} q^{\# \; \text{clusters}},$$ 
where the \emph{edge-weight} $p\in[0,1]$ and the \emph{cluster-weight} 
$q\in(0,\infty)$ are the parameters of the model. For $q\geq 1$, this 
model can be extended to infinite-volume lattices where it exhibits a 
phase transition at some critical parameter $p_c(q)$ (depending on the 
lattice). There are no general conjectures for the value of the critical 
point.

However, in the case of planar graphs, there is a connection (related to 
the Kramers-Wannier duality for the Ising model \cite{KramersWannier}) 
between random-cluster models on a graph and on its dual with the same 
cluster-weight $q$ and appropriately related edge-weights $p$ and 
$p^\star=p^\star(p)$. This relation leads in the particular case of 
$\mathbb Z^2$ (which is isomorphic to its dual) to a natural conjecture: 
the critical point is the same as the so-called \emph{self-dual point} 
satisfying $p_{sd}=p^\star(p_{sd})$, which has a known value
$$p_{sd}(q) = \frac{\sqrt{q}}{1+\sqrt{q}}.$$

\bigskip

In the present article, we prove this conjecture for all $q\geq 1$:
\begin{theorem}\label{critical_value}
  Let $q\geq 1$. The critical point $p_c=p_c(q)$ for the random-cluster 
  model with cluster-weight $q$ on the square lattice $\Z^2$ satisfies 
  $$p_c = \frac{\sqrt{q}}{1+\sqrt{q}}.$$
\end{theorem}

A rigorous derivation of the critical point was previously known in 
three cases. For $q=1$, the model is simply bond percolation, proved by 
Kesten in 1980~\cite{Kesten} to be critical at $p_c(1)=1/2$. For $q=2$, 
the self-dual value corresponds to the critical temperature of the Ising 
model, as first derived by Onsager in 1944~\cite{Onsager}; one can 
actually couple realizations of the Ising and FK models to relate their 
critical points, see \cite{Grimmett} and references therein for details.  
For modern proofs in that case, see \cite{AizenmanBarskyFernandez} or 
the short proof of \cite{BeffaraDuminilSmirnov}. Finally, for 
sufficiently large $q$, a proof is known based on the fact that the 
random-cluster model exhibits a first order phase transition 
(see~\cite{LaanaitMessagerMiracleSoleRuizShlosman, LaanaitMessagerRuiz}, 
the proofs are valid for $q$ larger than $25.72$). We mention that 
physicists derived the critical temperature for the Potts models with 
$q\geq4$ in 1978, using non-geometric arguments based on analytic 
properties of the Hamiltonian~\cite{HintermannKunzWu}.

In the sub-critical phase, we prove that the probability for two points 
$x$ and $y$ to be connected by a path decays exponentially fast with 
respect to the distance between $x$ and $y$. In the super-critical 
phase, the same behavior holds in the dual model. This phenomenon is 
known as a \emph{sharp phase transition}:
\begin{theorem}\label{exponential_decay}
  Let $q\geq1$. For any $p<p_c(q)$, there exist $0<C(p,q),c(p,q)<\infty$ 
  such that for any $x,y\in \Z^2$, \begin{equation}
    \ph(x\leftrightarrow y)\leq C(p,q)\e^{-c(p,q)|x-y|},
  \end{equation}
  where $|\cdot|$ denotes the Euclidean norm.
\end{theorem}

\bigskip

The proof involves two main ingredients. The first one is an estimate on 
crossing probabilities at the self-dual point $p=\sqrt{q}/(1+\sqrt{q})$: 
the probability of crossing a rectangle with aspect ratio $(\alpha,1)$ 
--- meaning that the ratio between the width and the height is of order 
$\alpha$ --- in the horizontal direction is bounded away from $0$ and $1$ 
uniformly in the size of the box.  This result is the main new 
contribution of this paper. It is a generalization of the celebrated 
Russo-Seymour-Welsh theorem for percolation.

The second ingredient is a collection of sharp threshold theorems, which 
were originally introduced for product measures. They have been used in 
many contexts, and are a powerful tool for the study of phase 
transitions, see Bollob\'as and Riordan 
\cite{BollobasRiordan,BollobasRiordanperco}. These theorems were later 
extended to positively associated measures by Graham and Grimmett 
\cite{GrahamGrimmett,GrahamGrimmett2,Grimmett}. In our case, they may be 
used to show that the probability of crossings goes to $1$ when 
$p>\sqrt{q}/(1+\sqrt{q})$. 

Actually, the situation is slightly more complicated than usual: the 
dependence inherent in the model makes boundary conditions difficult to 
handle, so that new arguments are needed. More precisely, one can use a 
classic sharp threshold argument for symmetric increasing events in 
order to deduce that the crossing probabilities of larger and larger 
domains, under \emph{wired boundary condition}, converge to $1$ whenever 
$p>\sqrt{q}/(1+\sqrt{q})$.  Moreover, the theorem provides us with 
bounds on the speed of convergence for rectangles with \emph{wired 
boundary condition}. A new way of combining long paths allows us to 
create an infinite cluster. We emphasize that the classic construction, 
used by Kesten \cite{Kesten} for instance, does not seem to work in our 
case. 

The approach allows the determination of the critical value, but it 
provides us with a rather weak estimate on the speed of convergence for 
crossing probabilities. Nevertheless, combining the fact that the 
crossing probabilities go to $0$ when $p<p_{sd}$ with a very general 
threshold theorem, we deduce that the cluster-size at the origin has 
finite moments of any order. It is then an easy step to derive the 
exponential decay of the two-point function.

\bigskip

Theorem~\ref{critical_value} has several notable consequences. First, it 
extends up to the critical point results that are known for the 
sub-critical random-cluster models under the exponential decay condition 
(for instance, Ornstein-Zernike estimates \cite{CampaninoIoffeVelenik} 
or strong mixing properties). Second, it identifies the critical value 
of the Potts models via the classical coupling between random-cluster 
models with cluster-weight $q\in \mathbb{N}$ and the $q$-state Potts 
models:
\begin{theorem}
  Let $q\geq2$ be an integer; consider the $q$-state Potts model on 
  $\Z^2$, defined by the Hamiltonian $$H_q(c) := - \sum_{(xy)\in E} 
  \delta_{c_x,c_y}$$ (where $c_x\in\{1,\ldots,q\}$ is the color at site 
  $x$ and $E$ the set of edges of the lattice). The model exhibits a 
  phase transition at the critical inverse temperature $$\beta_c(q)=\log 
  (1+\sqrt q).$$
\end{theorem}

\bigskip

The methods of this paper harness symmetries of the graph, together with 
the self-dual property of the square lattice. In the case of the 
hexagonal and triangular lattices, the symmetries of the graphs, the 
duality property between the hexagonal and the triangular lattices and 
the star-triangle relation allow us to extend the crossing estimate 
proved in Section~\ref{sec:crossings}, at the price of additional 
technical difficulties. The rest of the proof can be carried over to the 
triangular and the hexagonal lattices as well, yielding the following 
result:
\begin{theorem}\label{triangular}
  The critical value $p_c=p_c(q)$ for the random-cluster model with 
  cluster-weight $q\geq 1$ satisfies
  $$\begin{array}{ll}
    y_c^3+3y_c^2-q=0  & \text{on the triangular lattice and}\\
    y_c^3-3qy_c-q^2=0 & \text{on the hexagonal lattice,}
  \end{array}$$
  where $y_c:=p_c/(1-p_c)$. Moreover, there is exponential decay in the 
  sub-critical phase.
\end{theorem}

\bigskip

There are many unanswered questions related to 
Theorem~\ref{critical_value}. First, the behavior at criticality is not 
well understood in the general case, and it seems that new techniques 
are needed. It is conjectured that the random-cluster model undergoes a 
second-order phase transition for $q\in(0,4)$ (it is further believed 
that the scaling limit is then conformally invariant), and a first-order 
phase transition when $q\in(4,\infty)$. Proving the above for every $q$ 
remains a major open problem.
 We mention that the random-cluster models with parameter $q=2$ 
 \cite{Smirnov} or $q$ very large 
 \cite{LaanaitMessagerMiracleSoleRuizShlosman, LaanaitMessagerRuiz} are 
 much better understood.  Second, the technology developed in the 
 present article relies heavily on the positive association property of 
 the random-cluster measures with $q\geq 1$. Our strategy does not 
 extend to random-cluster models with $q<1$. Understanding these models 
 is a challenging open question.

\bigskip

The paper is organized as follows. In Section~\ref{sec:basic} we review 
some basic features of random-cluster models.  
Section~\ref{sec:crossings} is devoted to the statement and the proof of 
the crossing estimates. In Section~\ref{sec:sharp_threshold}, we briefly 
present the theory of sharp threshold that we will employ in the next
section.  Section~\ref{sec:proofs} contains the proofs of 
Theorems~\ref{critical_value} and~\ref{exponential_decay}.  
Section~\ref{sec:triangular} is devoted to extensions to other lattices 
and contains the proof of Theorem~\ref{triangular}.

\section{Basic features of the model}
\label{sec:basic}

We start with an introduction to the basic features of random-cluster 
models; more detail and proofs can be found in Grimmett's 
monograph~\cite{Grimmett}.

\paragraph{Definition of the random-cluster model.}

The random-cluster measure can be defined on any graph. However, we will 
restrict ourselves to the square lattice (of mesh size $1$), or more 
precisely a version rotated by an angle $\pi/4$, see 
Figure~\ref{fig:lattice_dual}. We denote this lattice by $\mathbb{L} = 
(\mathbb{V},\mathbb{E})$, with $\mathbb{V}$ denoting the set of 
\emph{sites} and $\mathbb{E}$ the set of \emph{edges}. In this paper, 
$G$ will always denote a connected subgraph of $\mathbb{L}$, \ie\  a 
subset of vertices of $\mathbb{V}$ together with all the edges between 
them. We denote by $\partial G$ the boundary of $G$, \ie\  the set of 
sites of $G$ linked by an edge of $\mathbb{E}$ to a site of 
$\mathbb{V}\setminus G$.

A \emph{configuration} $\omega$ on $G$ is a subgraph of $G$, composed of 
the same sites and a subset of its edges. We will call the edges 
belonging to $\omega$ \emph{open}, the others \emph{closed}. Two sites 
$a$ and $b$ are said to be \emph{connected} if there is an \emph{open 
path}, \ie\ a path composed of open edges only, connecting them (this 
event will be denoted by $a\leftrightarrow b$). Two sets $A$ and $B$ are 
\emph{connected} if there exists an open path connecting them (denoted 
$A\leftrightarrow B$). The maximal connected components will be called 
\emph{clusters}. We will often simply use the term \emph{path} for 
\emph{open path} when there is no possible ambiguity.

A \emph{boundary condition} $\xi$ is a partition of $\partial G$. We 
denote by $\omega \cup \xi$ the graph obtained from the configuration 
$\omega$ by identifying (or \emph{wiring}) the edges in $\xi$ that 
belong to the same component of $\xi$. Boundary conditions should be 
understood as encoding how sites are connected outside $G$. Let 
$o(\omega)$ (resp. $c(\omega)$) denote the number of open (resp.\ 
closed) edges of $\omega$ and $k(\omega,\xi)$ the number of connected 
components of $\omega\cup\xi$.  The probability measure 
$\phi^{\xi}_{p,q,G}$ of the random-cluster model on $G$ with parameters 
$p$ and $q$ and boundary condition $\xi$ is defined by
\begin{equation}
  \label{probconf}
  \phi_{p,q,G}^{\xi} (\left\{\omega\right\}) :=
  \frac {p^{o(\omega)}(1-p)^{c(\omega)}q^{k(\omega,\xi)}}
  {Z_{p,q,G}^{\xi}}
\end{equation}
for every configuration $\omega$ on $G$, where $Z_{p,q,G}^{\xi}$ is a 
normalizing constant referred to as the \emph{partition function}. 

Let $t<x$ and $y<z$; we will identify the rectangle $[t,x)\times[y,z)$ 
with the set of vertices in $\mathbb{V}$ that lie within it. The graph 
with vertex set $[t,x)\times[y,z)$ together with the induced subset of 
$\mathbb{E}$ is called a \emph{rectangle of \,$\mathbb{L}$}.

\begin{figure}[ht!]
    \begin{center}
     \includegraphics[width=\hsize]{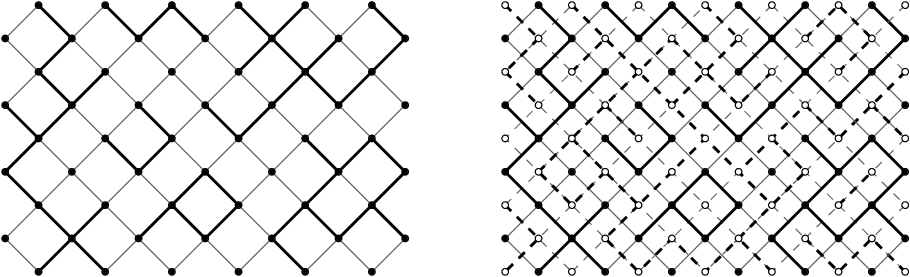}
    \end{center}
    \caption{\textbf{Left:} Example of a configuration on the rotated 
    lattice.  \textbf{Right:} A configuration together with its dual 
    configuration.}
    \label{fig:lattice_dual}
  \end{figure}

\paragraph{The finite energy property.}

This is a very simple property of random-cluster models. Let 
$\epsilon\in(0,1/2)$. The conditional probability for an edge to be 
open, knowing the states of all the other edges, is bounded away from 
$0$ and $1$ uniformly in $p\in(\epsilon,1-\epsilon)$ and in the 
configuration away from this edge. This property extends to any finite 
family of edges.

\paragraph{The domain Markov property.}

One can encode, using appropriate boundary condition $\xi$, the 
influence of the configuration outside a sub-graph on the measure within 
it.  Consider a graph $G=(V,E)$ and a random-cluster measure 
$\phi^{\psi}_{p,q,G}$ on it.  For $F\subset E$, consider $G'$ with $F$ 
as the set of edges and the endpoints of it as the set of sites. Then, 
the restriction to $G'$ of \smash{$\phi^{\psi}_{p,q,G}$} conditioned to 
match some configuration $\omega$ outside $G'$ is exactly 
\smash{$\phi_{p,q,G'}^{\xi}$}, where $\xi$ describes the connections 
inherited from $\omega\cup \psi$ (two sites are wired if they are 
connected by a path in $\omega\cup\psi$ outside the box).

\paragraph{The FKG inequality and comparison between boundary 
conditions.}

An event is called \emph{increasing} if it is preserved by addition of 
open edges, see~\cite{Grimmett}. Random-cluster models with parameter 
$q\geq 1$ are \emph{positively correlated}. This property has two 
important consequences,  the first one being the 
\emph{Fortuin-Kasteleyn-Ginibre inequality}:
\begin{equation}\label{FKG}
  \phi_{p,q,G}^{\xi}(A\cap B)\geq 
  \phi_{p,q,G}^{\xi}(A)\phi_{p,q,G}^{\xi}(B),
\end{equation}
which holds for every pair of increasing events $A$ and $B$ and any 
boundary conditions $\xi$. This correlation inequality is extremely 
important in the study of random-cluster models.

The second property is a \emph{comparison between boundary conditions}: 
for any boundary conditions $\psi\leq\xi$, meaning that sites wired in 
$\psi$ are wired in $\xi$, we have
\begin{equation}\label{comparison_between_boundary_conditions}
  \phi^{\psi}_{p,q,G}(A)\leq \phi^{\xi}_{p,q,G}(A)
\end{equation}
for any increasing event $A$. We say that $\phi_{p,q,G}^\xi$ 
\emph{stochastically dominates} $\phi_{p,q,G}^\psi$. Combined with the 
domain Markov property, the comparison between boundary conditions 
allows to give bounds on conditional probabilities.

\paragraph{Examples of boundary conditions: free, wired and periodic.}

Two boundary conditions play a special role in the study of the
random-cluster model. The \emph{wired} boundary condition, denoted by 
$\phi_{p,q,G}^1$, is specified by the fact that all the vertices on the 
boundary are pairwise wired (only one set in the partition). The 
\emph{free} boundary condition, denoted by $\phi_{p,q,G}^0$, is 
specified by no wiring between sites. These boundary conditions are 
extremal for the stochastic ordering, since any boundary condition has 
fewer (resp.\ more) wired sites than in the wired (resp.\ free) boundary 
condition.

We will also consider \emph{periodic} boundary conditions: for $n\geq1$ (not necessarily integer), 
the torus of size $n$ can be seen as the box $[0,n)^2$ with the boundary 
condition obtained by imposing that $(i,0)$ is wired to $(i,n)$ for 
every $i\in[0,n]$ and that $(0,j)$ is connected to $(n,j)$ for every 
$j\in[0,n]$.  We will denote the random-cluster measure on the torus of 
size $n$ by $\phi_{p,q,[0,n]^2}^{\rm p}$ or more concisely 
$\phi_{p,q,n}^{\rm p}$.  Note that this realization of the torus 
provides us with a natural embedding in the plane (though of course the 
boundary condition cannot be realized using disjoint paths outside the 
square $[0,n]^2$ because the torus itself is not a planar graph).

\paragraph{Dual graph and planar duality.} 

In two dimensions, one can associate with any random-cluster model a
dual model. Let $G$ be a finite graph embedded in a surface.  Define the
dual graph $G^{\star}=(V^{\star},E^{\star})$ in the usual way as
follows: place a dual site at the centers of the faces of $G$ (the
external face, when considering a graph on the plane, must be counted as
a face of the graph), and for every bond $e\in E$, place a dual bond
between the two dual sites corresponding to faces bordering $e$. Given a
subgraph configuration $\omega$, construct a bond model on $G^\star$ by
declaring any bond of the dual graph to be open (resp.\ closed) if the
corresponding bond of the primal lattice is closed (resp.\ open) for the
initial configuration. The new configuration is called the \emph{dual
  configuration} of $\omega$.

When defining the dual of the FK model, one must be careful about
boundary conditions (it will be crucial in this article). We first
recall the classic case: consider the random-cluster measure with
parameters $(p,q)$ on the square of size $n$, with wired boundary
conditions --- which can be realized as an FK model on a slightly larger
square, conditioned to have all the bonds outside the smaller square
open. The dual model on the dual graph (which is a square with an
additional outer vertex) given by the dual configurations then
corresponds to a random-cluster measure with free boundary conditions,
with the same parameter $q$ and a dual parameter $p^\star =
p^\star(p,q)$ satisfying $$p^\star(p,q) := \frac{(1-p)q}{(1-p)q+p}, \;
\text{or equivalently} \; \frac{p^\star p}{(1-p^\star)(1-p)}=q.$$ In
other words, the dual measure $(\phi_{p,q,n}^{1})^{\star}$ of
$\phi_{p,q,n}^{1}$ is $\phi_{p^{\star},q,n-1}^{0}$.  This relation is an
instance of \emph{planar duality}. It is then natural to define the
self-dual point $p_{sd}=p_{sd}(q)$ by solving the equation
$p^\star(p_{sd},q)=p_{sd}$, thus obtaining
\begin{equation}
  p_{sd}(q)=\frac{\sqrt{q}}{1+\sqrt{q}}.
\end{equation}
Similarly, the dual of a random-cluster model with parameters $(p,q)$
and free boundary conditions is a random-cluster model with parameters
$(p^\star,q)$ and wired boundary conditions.

\paragraph{Planar duality for periodic boundary conditions.}

The case of periodic boundary conditions, or equivalently the case of
the random-cluster model defined on a torus (with no boundary condition)
is a little more involved: indeed, its dual is \emph{not} a
random-cluster model; but it is not very different from one, and that
will be enough for our purposes. To state duality in this case, we need
additional notations. Recall that if $\omega$ is a configuration,
$o(\omega)$ stands for the number of open bonds in $\omega$, $c(\omega)$
for the number of closed bonds and $k(\omega)$ for the number of
connected components of $\omega$; let $f(\omega)$ be the number of faces
delimited by $\omega$, \ie\ the number of connected components of the
complement of the set of open bonds, and $s(\omega)$ be the number of
vertices in the underlying graph (it does not depend on $\omega$). We
will now define an additional parameter $\delta(\omega)$.

Call a (maximal) connected component of $\omega$ a \emph{net} if it
contains two non-contractible simple loops of different homotopy
classes, and a \emph{cycle} if it is non-contractible but is not a
net. Notice that every configuration $\omega$ can be of one of three
types:
\begin{itemize}
\item One of the clusters of $\omega$ is a net. Then no other cluster of
  $\omega$ can be a net or a cycle. In that case, we let
  $\delta(\omega)=2$;
\item One of the clusters of $\omega$ is a cycle. Then no other cluster
  can be a net, but other clusters can be cycles as well (in which case
  all the involved, simple loops are in the same homotopy class). We
  then let $\delta(\omega)=1$;
\item None of the clusters of $\omega$ is a net or a cycle. We let
  $\delta(\omega)=0$.
\end{itemize}
With this additional notation, Euler's formula becomes
\begin{equation}
  \label{eq:euler}
  s(\omega) - o(\omega) + f(\omega) = k(\omega) + 1 - \delta(\omega).
\end{equation}
Besides, these terms transform in a simple way under duality:
$o(\omega)+o(\omega^\star)$ is a constant, $f(\omega) = k(\omega^\star)$
and $\delta(\omega) = 2 - \delta(\omega^\star)$. The same proof as that
of usual duality, taking the additional topology into account, then
leads to the relation
\begin{equation}
  \label{eq:dualperio}
  (\phi_{p,q,n}^{\mathrm p})^\star(\{\omega\}) \propto
  q^{1-\delta(\omega)} \phi_{p^\star,q,n}^{\mathrm p} (\{\omega\}).
\end{equation}
This means that even though the dual model of the periodic boundary
condition FK model is not exactly an FK model at the dual parameter, it
is absolutely continuous with respect to it and the Radon-Nikodym
derivative is bounded above and below by constants depending only on
$q$. Another way of stating the same result would be to define a
\emph{balanced FK model} with weights $$\tilde\phi_{p,q,n}^{\mathrm p}
(\{\omega\}) = \frac {(\sqrt q)^{1-\delta(\omega)}} {Z} \phi_{p,q,n}
^{\mathrm p} (\{\omega\}):$$ this one is absolutely continuous with
respect to the usual FK model and does satisfy exact duality.

\paragraph{Infinite-volume measures and the definition of the critical 
point.}

The domain Markov property and comparison between boundary conditions
allow us to define infinite-volume measures. Indeed, one can consider a
sequence of measures on boxes of increasing size with free boundary
conditions. This sequence is increasing in the sense of stochastic
domination, which implies that it converges weakly to a limiting
measure, called the random-cluster measure on $\mathbb{L}$ with free
boundary condition (and denoted by $\ph^0$). This classic construction
can be performed with many other sequences of measures, defining several
\emph{a priori} different infinite-volume measures on $\mathbb{L}$. For
instance, one can define the random-cluster measure $\ph^1$ with wired
boundary condition, by considering the decreasing sequence of
random-cluster measures on finite boxes with wired boundary condition.

For given $q\geq 1$, it is known that uniqueness can fail only for $p$ 
in a countable set $\mathcal{D}_q$, see Theorem (4.60) of 
\cite{Grimmett}.  Therefore, there exists a \emph{critical point} $p_c$ 
such that for any infinite-volume measure with $p<p_c$ (resp.\ $p>p_c$), 
there is almost surely no infinite component of connected sites (resp.\ 
at least one infinite component).

\begin{remark}\label{conjecture_uniqueness}
  It is natural to conjecture that the critical point satisfies 
  $p_c=p_{sd}$. Indeed, if one assumes $p_c\neq p_{sd}$, there would be 
  two phase transitions, one at $p_c$, due to the change of behavior in 
  the primal model, and one at $p_c^\star$, due to the change of 
  behavior in the dual model.
\end{remark}

\paragraph{The inequality $p_c\geq p_{sd}$.}

As in the case of percolation, a lower bound for the critical value can 
be derived using the uniqueness of the infinite cluster above the 
critical point.  Indeed, if one assumes that $p_c<p_{sd}$, the 
configuration at $p_{sd}$ must contain one infinite open cluster and one 
infinite dual open cluster (since the random-cluster model in the dual 
is then super-critical as well). Intuition indicates that such 
coexistence would imply that there is more than one infinite open 
cluster; an elegant argument (due to Zhang in the case of percolation) 
formalizes this idea.  We refer to the exposition in Theorem (6.17) 
of~\cite{Grimmett} for full detail, but still give a sketch of the 
argument.

The proof goes as follows, see Figure~\ref{fig:Zhang}. Assume that 
$p_c<p_{sd}$ and consider the random-cluster model with $p=p_{sd}$.  
There is an infinite open cluster, and therefore, we can choose a large 
box such that the infinite open cluster and the dual infinite open 
cluster touch the boundary with probability greater than $1-\ep$.  The 
FKG inequality (through the so-called ``square-root trick'': for two increasing events $A$ and $B$ with same probability, $\phi^{\xi}_{p,q,G}(A\cap B)\geq 1-(1-\phi^{\xi}_{p,q,G}(A))^{1/2}$) implies 
that the infinite open cluster actually touches the top side of the box, 
using only edges outside the box, with probability greater than 
$1-\ep^{1/4}$. We deduce that with probability at least $1-2\ep^{1/4}$, 
the infinite open cluster touches both the top and bottom sides, using 
only edges outside of the box.

A similar argument implies that the infinite dual open cluster touches 
both the left- and right-hand sides of the box with probability at least 
$1-2\ep^{1/4}$.  Therefore, with probability at least $1-4\ep^{1/4}$, 
the complement of the box contains an infinite open path touching the 
top of the box, one touching the bottom, and infinite dual open paths 
touching each of the vertical edges.  Enforcing edges in the box to be 
closed, which brings only a positive multiplicative factor due to the 
finite energy property of the model, and choosing $\ep$ sufficiently 
small, we deduce that there are two infinite open clusters with positive 
probability.  Since the infinite open cluster must be unique (see \cite{Grimmett} again), this is a 
contradiction, which implies that $p_c\geq p_{sd}$.

\begin{figure}[ht!]
    \begin{center}
    \includegraphics[width=1.0\hsize]{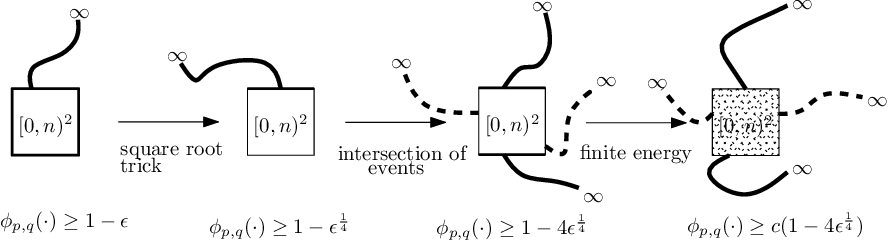} \end{center}
    \caption{The steps of the proof that $p_c\geq p_{sd}$.}
    \label{fig:Zhang}
  \end{figure}
\bigskip

When $p<p_{sd}\leq p_c$, there is no infinite cluster for any 
infinite-volume measure. General arguments imply uniqueness of the 
infinite-volume measure whenever $p\neq p_{sd}$ and $q\geq 1$ (see 
Theorem~(6.17) of \cite{Grimmett}). This fact will be useful in the 
sequel since, \emph{except at criticality}, we do not have to specify 
which infinite-volume measure is under consideration. We will denote the 
unique infinite-volume measure by $\phi_{p,q}$ when $p\neq p_{sd}$.

\section{Crossing probabilities for rectangles at the self-dual point}
\label{sec:crossings}

In this section, we prove crossing estimates for rectangles of 
prescribed aspect ratio. This is an extension of the Russo-Seymour-Welsh 
theory for percolation. We will work with $p=p_{sd}(q)$ and the measures 
$\phi_{p_{sd},q}^1$ and $\phi_{p_{sd},q,n}^{\rm p}$; we present the 
proof in the periodic case. The case of the (bulk) wired boundary 
condition can be derived from this case (see Corollary~\ref{RSW_bulk}).

For a rectangle $R$, let $\C_v(R)$ denote the event that there exists a 
path between the top and the bottom sides which stays inside the 
rectangle. Such a path is called a \emph{vertical (open) crossing} of 
the rectangle. Similarly, we define $\C_h$ to be the event that there 
exists an \emph{horizontal open crossing} between the left and the right 
sides.  Finally, $\C_v^\star(R^\star)$ denotes the event that there 
exists a dual-open crossing from top to bottom in the dual graph 
$R^\star$ of $R$.

The following theorem states that, at the self-dual point, the 
probability of crossing a rectangle horizontally is bounded away from 
$0$ uniformly in the sizes of both the rectangle and the torus provided 
that the aspect ratio of the rectangles remains constant. The size of 
the ambient torus is denoted by $m$. Note that $p=p^\star$ when 
$p=p_{sd}$, and hence the balanced FK measure on the torus is 
self-dual.

\begin{theorem}\label{RSW_torus}
  Let $\al>1$ and $q\geq 1$. There exists $c(\al)>0$ such that for every 
  $m>\alpha n>0$,
  \begin{equation}
    \phi_{p_{sd},q,m}^{\rm p}\big(\C_h([0,\al n)\times[0,n))\big)\geq 
    c(\al).
  \end{equation}
\end{theorem}

We begin the proof with a lemma, which corresponds to the existence of 
$c(1)$ and is based on the self-duality of random-cluster measures on 
the torus.  This lemma is classic and is the natural starting point for 
any attempt to prove RSW-like estimates.

\begin{lemma}\label{square_crossings}
  Let $q\geq 1$, there exists $c(1)>0$ (depending only on the parameter
  $q$) such that for every $m>n\geq 1$, $\displaystyle
  \phi_{p_{sd},q,m}^{\rm p}\big(\C_h([0,n)^2)\big)\geq c(1)$.
\end{lemma}

\begin{figure}[ht!]
  \begin{center}
    \includegraphics[width=\hsize]{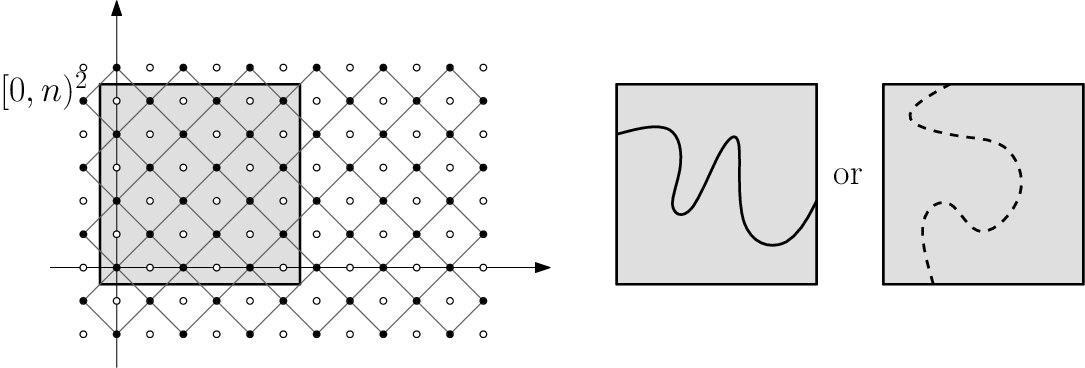}
  \end{center}
  \caption{\textbf{Left:} The square $[0,n)^2$ (all the sites in the shaded region) and its dual have the 
  same graph structure.  \textbf{Right:} The events $\C_h([0,n)^2)$ and 
  $\C_v^\star([0,n)^2)$.}
  \label{fig:square_symmetry}
\end{figure}

\begin{proof}
  Note that the dual of $[0,n)^2$ is $[0,n)^2$ (meaning the sites of the 
  dual torus inside $[0,n)^2$), see Figure~\ref{fig:square_symmetry}. If 
  there is no open crossing from left to right in $[0,n)^2$, there 
  exists necessarily a dual-open crossing from top to bottom in the dual 
  configuration. Hence, the complement of $\C_h([0,n)^2)$ is 
  $\C_v^\star([0,n)^2)$, thus yielding  \begin{equation*}
    \phi_{p_{sd},q,m}^{\rm 
    p}\big(\C_h([0,n)^2)\big)+\phi_{p_{sd},q,m}^{\rm 
    p}\big(\C_v^\star([0,n)^2)\big) = 1.
  \end{equation*}
  Using the duality property for periodic boundary conditions and the
  symmetry of the lattice, the probability $\phi_{p_{sd},q,m}^{\rm
    p}(\C_v^\star([0,n)^2))$ is larger than $c\phi_{p_{sd},q,m}^{\rm
    p}(\C_h([0,n)^2))$ (for some constant $c$ only depending on $q$),
  giving
  \begin{equation*}
    1\leq (1+c)\phi_{p_{sd},q,m}^{\rm p}\big(\C_h([0,n)^2)\big),
  \end{equation*}
  which concludes the proof.
\end{proof}

The only major difficulty is to prove that rectangles of aspect ratio
$\al$ are crossed in the horizontal direction --- with probability
uniformly bounded away from $0$ --- for \emph{some} $\al>1$. There are
many ways to prove this in the case of percolation. Nevertheless, they
always involve independence in a crucial way; in our case, independence
fails, so we need a new argument. The main idea is to invoke
self-duality in order to enforce the existence of crossings, even in the
case where boundary conditions could look disadvantageous. In order to
do that, we introduce the following family of domains, which are in some
sense nice symmetric domains.

\begin{figure}[ht!]
  \begin{center}
    \includegraphics[width=0.6\hsize]{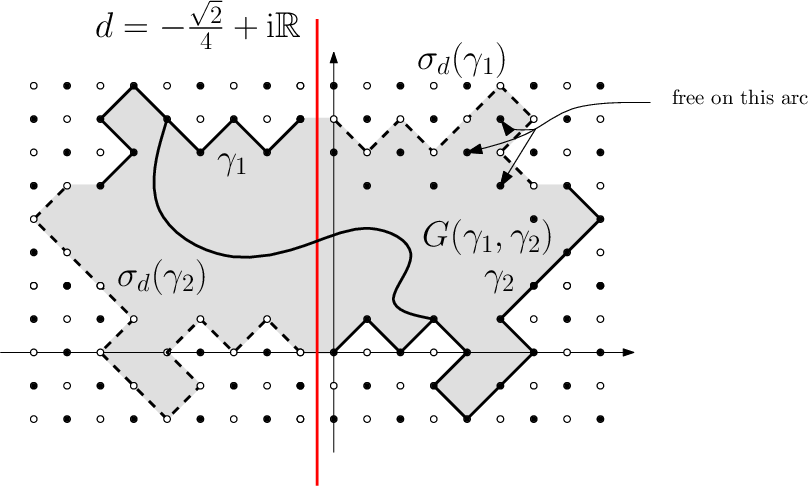}
  \end{center}
  \caption{Two paths $\g_1$ and $\g_2$ satisfying Hypothesis $(\star)$ 
  and the graph $G(\g_1,\g_2)$.}
  \label{fig:symmetric_domain}
\end{figure}

Define the line $d:=-\sqrt{2}/4+{\rm i}\R$. The orthogonal symmetry 
$\sigma_d$ with respect to this line maps $\mathbb{L}$ to 
$\mathbb{L}^\star$. Let $\g_1$ and $\g_2$ be two paths satisfying the 
following Hypothesis $(\star)$, see Figure~\ref{fig:symmetric_domain}:
\begin{itemize}
  \item $\g_1$ remains on the left of $d$ and $\g_2$ remains on the 
    right;
  \item $\g_2$ begins at $0$ and $\g_1$ begins on a site of 
    $\mathbb{L}\cap (-\sqrt 2/ 2+{\rm i}\mathbb{R}_+)$;
  \item $\g_1$ and $\sigma_d(\g_2)$ do not intersect (as curves in the 
    plane);
  \item $\g_1$ and $\sigma_d(\g_2)$ end at two sites (one primal and one 
    dual) which are at distance $\sqrt2/2$ from each other.
\end{itemize}
The definition extends trivially via translation, so we will say that 
the pair $(\g_1,\g_2)$ satisfies Hypothesis $(\star)$ if one of its 
translations does.

When following the paths in counter-clockwise order, we can create a
circuit by linking the end points of $\g_1$ and $\sigma_d(\g_2)$ by a
straight line, the start points of $\sigma_d(\g_2)$ and $\g_2$, the end
points of $\g_2$ and $\sigma_d(\g_1)$, and the start points of
$\sigma_d(\g_1)$ and $\g_1$. The circuit $(\g_1,\sigma_d(\g_2), \g_2,
\sigma_d(\g_1))$ surrounds a set of vertices of $\mathbb{L}$. Define the
graph $G(\g_1,\g_2)$ composed of sites of $\mathbb{L}$ that are
surrounded by the circuit $(\g_1,\sigma_d(\g_2),\g_2,\sigma_d(\g_1))$,
and of edges of $\mathbb{L}$ that remain entirely within the circuit
(boundary included).

The \emph{mixed boundary condition} on this graph is wired on $\g_1$ 
(all the edges are pairwise connected), wired on $\g_2$, and free 
elsewhere.  We denote the measure on $G(\g_1,\g_2)$ with parameters 
$(p_{sd},q)$ and mixed boundary condition by $\phi_{p_{sd},q,\g_1,\g_2}$ 
or more simply $\phi_{\g_1,\g_2}$.

\begin{lemma}\label{self-duality}
  For any pair $(\g_1,\g_2)$ satisfying Hypothesis $(\star)$, the 
  following estimate holds: $$\phi_{\g_1,\g_2}(\g_1\leftrightarrow 
  \g_2)\geq\frac{1}{1+q^2}.$$
\end{lemma}

\begin{proof}
  On the one hand, if $\g_1$ and $\g_2$ are not connected, 
  $\sigma_d(\g_1)$ and $\sigma_d(\g_2)$ must be connected by a dual path 
  in the dual model (event corresponding to 
  $\sigma_d(\g_1)\leftrightarrow\sigma_d(\g_2)$ in the dual model).  
  Hence,
  \begin{equation}\label{self-duality1}
    1 = \phi_{\g_1,\g_2} (\g_1\leftrightarrow\g_2) + 
    \sigma_d*\phi^\star_{\g_1,\g_2} (\g_1\leftrightarrow\g_2),
  \end{equation}
  where $\sigma_d*(\phi_{\g_1,\g_2}^\star)$ denotes the image under 
  $\sigma_d$ of the dual measure of $\phi_{\g_1,\g_2}$. This
  measure lies on $G(\g_1,\g_2)$ as well and has parameters 
  $(p_{sd},q)$. 

  When looking at the dual measure of a random-cluster model, the 
  boundary condition is transposed into a new boundary condition for the 
  dual measure. In the case of the periodic boundary condition, we 
  obtained the same boundary condition for the dual measure. Here, the 
  boundary condition becomes wired on $\g_1\cup \g_2$ and free elsewhere 
  (this is easy to check using Euler's formula).

  It is very important to notice that the boundary condition is 
  \emph{not} exactly the mixed one, since $\g_1$ and $\g_2$ are wired 
  \emph{together}.  Nevertheless, the Radon-Nikodym derivative of 
  $\sigma_d*\phi_{\g_1,\g_2}^\star$ with respect to $\phi_{\g_1,\g_2}$ 
  is easy to bound. Indeed, for any configuration $\omega$, the number 
  of cluster can differ only by $1$ when counted in 
  $\sigma_d*\phi_{\g_1,\g_2}^\star$ or $\phi_{\g_1,\g_2}$ so that the 
  ratio of partition functions belongs to $[1/q,q]$. Therefore, the 
  ratio of probabilities of the configuration $\omega$ remains between 
  $1/q^2$ and $q^2$. This estimate extends to events by summing over all 
  configurations. Therefore, $$\sigma_d*\phi_{\g_1,\g_2}^\star 
  (\g_1\leftrightarrow\g_2) \leq q^2 \phi_{\g_1,\g_2} 
  (\g_1\leftrightarrow\g_2).$$ When plugging this inequality into 
  \eqref{self-duality1}, we obtain $$\phi_{\g_1,\g_2} 
  (\g_1\leftrightarrow\g_2) + q^2 \phi_{\g_1,\g_2} 
  (\g_1\leftrightarrow\g_2) \geq 1$$ which implies the claim.
\end{proof}

We are now in a position to prove the key result of this section.

\begin{proposition}\label{box_crossing}
  For all $m>3n/2>0$, the following holds: $$\phi_{p_{sd},q,m}^{\rm p} 
  \big[ \C^v\big([0,n)\times[0,3/2n)\big) \big] \geq 
  \frac{c(1)^3}{2(1+q^2)}.$$
\end{proposition}

Before proving this proposition, we show how it implies 
Theorem~\ref{RSW_torus}. The strategy is straightforward and classic: we 
combine crossings together, using only the FKG inequality.

\begin{proof}[Proof of Theorem~\ref{RSW_torus}]
  If $\al<3/2$, Proposition~\ref{box_crossing} implies the claim so we
  can assume $\al>3/2$. Define the following rectangles, see
  Figure~\ref{fig:a}: $$R^h_j = [j n/2,jn/2+3n/2) \times [0,n) \quad
  \text{and} \quad R^v_j=[jn/2,j n/2+n)\times[0,n)$$ for
  $j\in[0,\lfloor2\alpha\rfloor-1]$, where $\lfloor x\rfloor$ denotes
  the integer part of $x$. If every rectangle $R^h_j$ is crossed
  horizontally, and every rectangle $R^v_j$ is crossed vertically, then
  $[0,\alpha n)\times[0,n)$ is crossed horizontally. We denote this
  event by $B$.  The rectangle $R^h_j$ is crossed horizontally with
  probability greater than $c(1)^3/[2(1+q^2)]$
  (Proposition~\ref{box_crossing}), the rectangle $R^v_j$ is crossed
  vertically with probability greater than $c(1)$
  (Lemma~\ref{square_crossings}) and so, using the FKG
  inequality, \begin{equation*} \phi_{p_{sd},q,m}^{\rm
      p}\big(\C_h([0,\al n)\times[0,n))\big)\geq \phi_{p_{sd},q,m}^{\rm
      p}(B)\geq
    \left(\frac{c(1)^4}{2(1+q^2)}\right)^{\lfloor2\alpha\rfloor}.
  \end{equation*}
The claim follows with $c(\alpha) := [c(1)^4/(2+2q^2)]
  ^{\lfloor2\alpha\rfloor}$.
\end{proof}

\begin{figure}[ht!]
  \begin{center}
    \includegraphics[width=0.5\hsize]{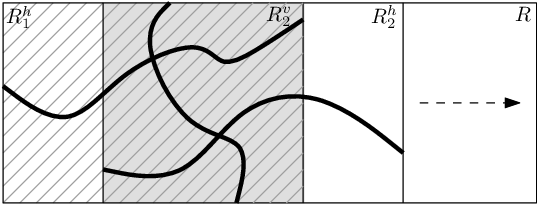}\quad \quad \quad 
    \includegraphics[width=0.3\hsize]{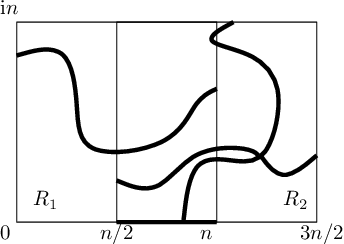}
  \end{center}
  \caption{\textbf{Left:} A combination of crossings in smaller 
  rectangles creating a horizontal crossing of a very long rectangle.  
  \textbf{Right:} The rectangles $R$, $R_1$ and $R_2$ and the event 
  $A$.}
  \label{fig:a}
\end{figure}

\begin{proof}[Proof of Proposition~\ref{box_crossing}]
  The proof goes as follows: we start by creating two paths crossing 
  square boxes, and we then prove that they are connected with good 
  probability.

  \paragraph{Setting of the proof.}

  Consider the rectangle $R=[0,3n/2)\times[0,n)$ which is the union of 
  the rectangles $R_1=[0,n)\times[0,n)$ and $R_2=[n/2,3n/2)\times[0,n)$, 
  see Figure~\ref{fig:a}. Let $A$ be the event defined by the following 
  conditions:
  \begin{itemize}
    \item $R_1$ and $R_2$ are both crossed horizontally (these events 
      have probability at least $c(1)$ to occur, using 
      Lemma~\ref{square_crossings});
    \item $[n/2,n)\times\{0\}$ is connected inside $R_2$ to the top side
      of $R_2$ (which has probability greater than $c(1)/2$ to occur
      using symmetry and Lemma~\ref{square_crossings}).
  \end{itemize}
  Employing the FKG inequality, we deduce that:
  \begin{equation}\label{abc}
    \phi_{p_{sd},q,m}^{\rm p}(A)\geq \frac{c(1)^3}{2}.
  \end{equation}
  When $A$ occurs, define $\Ga_1$ to be the top-most horizontal crossing 
  of $R_1$, and $\Ga_2$ the right-most vertical crossing of $R_2$ from
  $[n/2,n)\times\{0\}$ to the top side. Note that this path is 
  automatically connected to the right-hand side of $R_2$ --- which is 
  the same as the right-most side of $R$. If $\Ga_1$ and $\Ga_2$ are 
  connected, then there exists a horizontal crossing of $R$. In the 
  following, we show that $\Ga_1$ and $\Ga_2$ are connected with good 
  probability.

  \paragraph{Exploration of the paths $\Ga_1$ and $\Ga_2$.}

  There is a standard way of exploring $R$ in order to discover $\Ga_1$ 
  and $\Ga_2$. Start an exploration from the top-left corner of $R$ that 
  leaves open edges on its right, closed edges on its left and remains 
  in $R_1$. If $A$ occurs, this exploration will touch the right-hand 
  side of $R_1$ before its bottom side; stop it the first time it does.  
  Note that the exploration process ``slides'' between open edges of the 
  primal lattice and dual open edges of the dual (formally, this 
  exploration process is defined on the medial lattice, see \emph{e.g.} 
  \cite{BeffaraDuminilSmirnov}). The open edges that are adjacent to the 
  exploration form the top-most horizontal crossing of $R_1$, \ie\  
  $\Ga_1$. At the end of the exploration, the process has \emph{a 
  priori} discovered a set of edges which lies above $\Ga_1$, so that 
  the remaining part of $R_1$ is undiscovered.

  By starting an exploration at point $(n,0)$, leaving open edges on its 
  left and closed edges on its right, we can explore the rectangle 
  $R_2$. If $A$ holds, the exploration ends on the top side of $R_2$.  
  The open edges adjacent to the exploration constitute the path $\Ga_2$ 
  and the set of edges already discovered lies ``to the right'' of 
  $\Gamma_2$.

  \begin{figure}[ht!]
    \begin{center}
      \includegraphics[width=0.8\hsize]{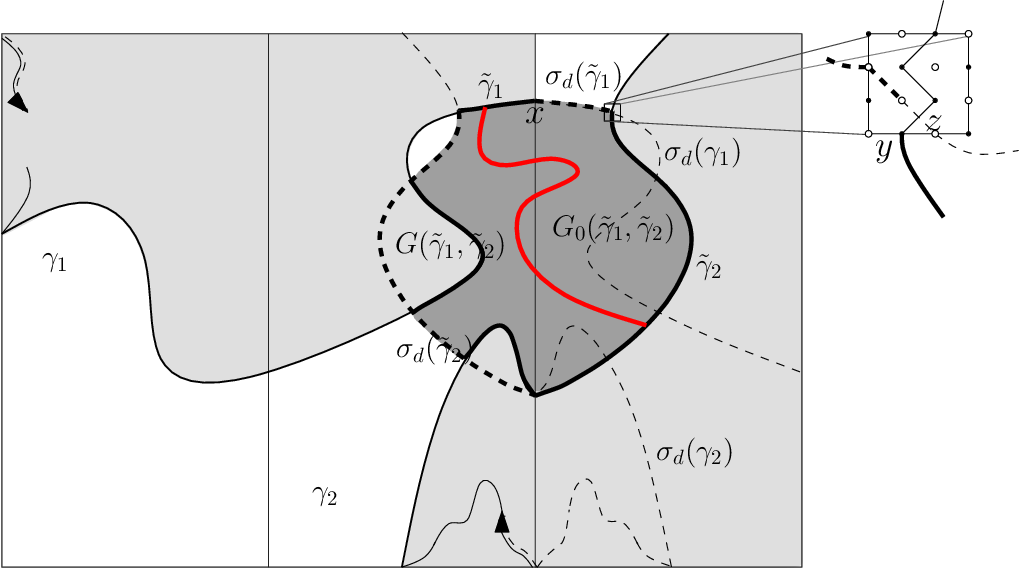}
    \end{center}
    \caption{The light gray area is the part of $R$ that is \emph{a 
    priori} discovered by the exploration processes (note that this area 
    can be much smaller). The dark gray is the domain 
    $G_0(\tilde{\g}_1,\tilde{\g}_2)$. We have depicted all the paths 
    involved in the construction. Note that dashed curves are ``virtual 
    paths''  of the dual lattice obtained by the reflection $\sigma_d$: 
    they are not necessarily dual open.}
    \label{fig:rectangle_R}
  \end{figure}

  \paragraph{The reflection argument.}

  Assume first that we know $\Ga_1=\g_1$ and $\Ga_2=\g_2$ and that they 
  do not intersect. Let $x$ be the end-point of $\g_1$, \ie\ its unique 
  point on the right-hand side of $R_1$. We want to define a set 
  $G_0(\g_1,\g_2)$ similar to those considered in 
  Lemma~\ref{self-duality}. Apply the following ``surgical procedure,'' 
  see Figure~\ref{fig:rectangle_R}:
  \begin{itemize}
    \item First, define the symmetric paths $\sigma_d(\g_1)$ and 
      $\sigma_d(\g_2)$ of $\g_1$ and $\g_2$ with respect to the line 
      $d:=(n-\sqrt{2}/4)+{\rm i}\R$;
    \item Then, parametrize the path $\sigma_d(\g_1)$ by the distance 
      (along the path) to its starting point $\sigma_d(x)$ and define 
      $\tilde{\g}_1\subset \g_1$ so that $\sigma_d(\tilde{\g}_1)$ is the 
      part of $\sigma_d(\g_1)$ between the start of the path and the 
      first time it intersects $\g_2$. As before, the paths are 
      considered as curves of the plane; we denote $z$ the intersection 
      point of the two curves. Note that $\g_1$ and $\g_2$ are not 
      intersecting, which forces $\sigma_d(\g_1)$ and $\g_2$ to be;
    \item From this, parametrize the path $\g_2$ by the distance to its 
      starting point $(n,0)$ and set $y$ to be the last visited site in 
      $\mathbb{L}$ before the intersection $z$. Define $\tilde{\g}_2$ to 
      be the part of $\g_2$ between the last point intersecting $n+{\rm 
      i}\R$ before $y$ and $y$ itself;
    \item Paths $\tilde{\g}_1$ and $\tilde{\g}_2$ satisfy Hypothesis 
      $(\star)$ so that the graph $G(\tilde{\g}_1,\tilde{\g}_2)$ can be 
      defined;
    \item Construct a sub-graph $G_0(\g_1,\g_2)$ of 
      $G(\tilde{\g}_1,\tilde{\g}_2)$ as follows: the edges are given by 
      the edges of $\mathbb{L}$ included in the connected component of 
      $G(\tilde{\g}_1,\tilde{\g}_2)\setminus (\g_1\cup\g_2)$ 
      (\emph{i.e.} $G(\tilde{\g}_1,\tilde{\g}_2)$ minus the set 
      $\g_1\cup\g_2$) containing $d$ (it is the connected component 
      which contains $x-\varepsilon {\rm i}$, where $\varepsilon$ is a 
      very small number), and the sites are given by their endpoints.
  \end{itemize}

  The graph $G_0(\g_1,\g_2)$ has a very useful property: none of its 
  edges has been discovered by the previous exploration paths.  Indeed, 
  $\sigma_d(\g_1)$ and $\sigma_d(x)$ are included in the unexplored 
  connected component of $R\setminus R_1$, and so does 
  $G_0(\g_1,\g_2)\cap (R\setminus R_1)$. Edges of $G_0(\g_1,\g_2)$ in 
  $R_1$ are in the same connected component of $R\setminus (\g_1\cup 
  \g_2)$ as $x-\epsilon {\rm i}$, and thus lie `below' $\g_1$.

  \paragraph{Conditional probability estimate.}

  Still assuming that $\g_1$ and $\g_2$ do not intersect, we would like 
  to estimate the probability of $\g_1$ and $\g_2$ being connected by a 
  path knowing that $\Ga_1=\g_1$ and $\Ga_2=\g_2$.  Following the 
  exploration procedure described above, we can discover $\g_1$ and 
  $\g_2$ without touching any edge in the interior of $G_0(\g_1,\g_2)$.  
  Therefore, the process in the domain is a random-cluster model with 
  specific boundary condition. 

  The boundary of $G_0(\g_1,\g_2)$ can be split into several sub-arcs of 
  various types (see Figure~\ref{fig:rectangle_R}): some are sub-arcs of 
  $\g_1$ or $\g_2$, while the others are (adjacent to) sub-arcs of their 
  symmetric images $\sigma_d(\g_1)$ and $\sigma_d(\g_2)$. The 
  conditioning on $\Ga_1=\g_1$ and $\Ga_2=\g_2$ ensures that the edges 
  along the sub-arcs of the first type are open; the connections along 
  the others depend on the exact explored configuration in a much more 
  intricate way, but in any case the boundary condition imposed on the 
  configuration inside $G(\tilde\g_1, \tilde\g_2)$ is larger than the 
  mixed boundary condition. Notice that \emph{any} boundary condition 
  dominates the free one and that $\tilde{\g}_1$ and $\tilde{\g}_2$ are 
  two sub-arcs of the first type (they are then wired). We deduce that 
  the measure restricted to $G_0(\tilde\g_1,\tilde\g_2)$ stochastically 
  dominates the restriction of $\phi_{\tilde\g_1,\tilde\g_2}$ to 
  $G_0(\tilde\g_1,\tilde\g_2)$.

  From these observations, we deduce that for any increasing event $B$ 
  \emph{depending only on edges in $G_0(\g_1,\g_2)$},
  \begin{equation}
    \phi^{\rm p}_{p_{sd},q,m}(B|\Ga_1=\g_1,\Ga_2=\g_2) \geq 
    \phi_{\tilde{\g}_1,\tilde{\g}_2}(B).
  \end{equation}
  In particular, we can apply this inequality to $\{\g_1\leftrightarrow 
  \g_2\text{ in }G_0(\g_1,\g_2)\}$. Note that if $\tilde{\g}_1$ and 
  $\tilde{\g}_2$ are connected in $G(\tilde{\g}_1,\tilde{\g}_2)$, $\g_1$ 
  and $\g_2$ are connected in $G_0(\tilde{\g_1},\tilde{\g}_2)$. The 
  first event is of $\phi_{\tilde{\g}_1,\tilde{\g}_2}$-probability 
  at least $1/(1+q^2)$, implying
  \begin{align}
    \phi^{\rm p}_{p_{sd},q,m}
    (\g_1\leftrightarrow\g_2|\Ga_1=\g_1,\Ga_2=\g_2) &\geq
    \phi_{\tilde{\g}_1,\tilde{\g}_2}
    (\g_1\leftrightarrow\g_2\text{~in~}G_0(\g_1,\g_2)) \nonumber \\
    & \geq \phi_{\tilde{\g}_1,\tilde{\g}_2}(\tilde{\g}_1\leftrightarrow
    \tilde{\g}_2)\geq\frac1{1+q^2}.\label{estimate_duality}
  \end{align}

  \paragraph{Conclusion of the proof.}

  Note the following obvious fact: if $\g_1$ and $\g_2$ intersect, the 
  conditional probability that $\Ga_1$ and $\Ga_2$ intersect, knowing 
  $\Ga_1=\g_1$ and $\Ga_2=\g_2$ is equal to $1$ --- in particular, it is 
  greater than $1/(1+q^2)$. Now,
  \begin{align*}
    \phi^{\rm p}_{p_{sd},q,m}(\C_h(R))&\geq \phi^{\rm
      p}_{p_{sd},q,m}(\C_h(R)\cap A)\\
    &\geq\phi^{\rm p}_{p_{sd},q,m}(\{\Ga_1\leftrightarrow\Ga_2\}\cap
    A)\\
    &=\phi^{\rm p}_{p_{sd},q,m}\big( \phi^{\rm
      p}_{p_{sd},q,m}(\Ga_1\leftrightarrow\Ga_2|\Ga_1,\Ga_2)\mathbbm{1}_{A}\big)\\
    &\geq \frac{1}{1+q^2}\phi(A)\geq \frac{c(1)^3}{2(1+q)^2}
  \end{align*}
  where the first two inequalities are due to inclusion of events, the 
  third one to the definition of conditional expectation, and the fourth 
  and fifth ones, to \eqref{estimate_duality} and \eqref{abc}.
\end{proof}

An equivalent of Theorem~\ref{RSW_torus} holds in the case of the 
infinite-volume random-cluster measure with wired boundary condition.

\begin{corollary}\label{RSW_bulk}
  Let $\al>1$ and $q\geq 1$; there exists $c(\al)>0$ such that for every 
  $n\geq1$,
  \begin{equation}
    \phi_{p_{sd},q}^1\big[\C_h\big([0,\al n)\times[0,n)\big)\big]\geq 
    c(\al).
  \end{equation}
\end{corollary}

\begin{proof}
  Let $\alpha>1$ and $m>2\alpha n>0$. Using the invariance under 
  translations of $\phi^{\rm p}_{p_{sd},q,m}$ and comparison between 
  boundary conditions, we have $$\phi^{1}_{p_{sd},q,[-\frac m2,\frac 
  m2)^2}\big[\C_h\big([0,\al n)\times[0,n)\big)\big]\geq \phi^{\rm 
  p}_{p_{sd},q,m}\big[\C_h\big([0,\al n)\times[0,n)\big)\big]\geq 
  c(\alpha).$$ When $m$ goes to infinity, the left hand side converges 
  to the probability in infinite volume, so that 
  \begin{equation*}
    \phi^{1}_{p_{sd},q}\big[\C_h\big([0,\al n)\times[0,n)\big)\big]\geq 
    c(\alpha).\qedhere
  \end{equation*}
\end{proof}

\begin{remark}
  The only place where we use the periodic and the (bulk) wired boundary 
  conditions is in the estimate of Lemma~\ref{square_crossings}. For 
  instance, if one could prove that the probability for a square box to 
  be crossed from top to bottom with free boundary conditions stays 
  bounded away from $0$ when $n$ goes to infinity, then an equivalent of 
  Theorem~\ref{RSW_torus} would follow with free boundary conditions.

  Uniform estimates with respect to boundary conditions should be true 
  for $q\in[1,4)$; we expect the random-cluster model to be conformally 
  invariant in the scaling limit. It should be false for $q\geq4$.  
  Indeed, for $q>4$, the phase transition is (conjecturally) of first 
  order in the sense that there should not be uniqueness of the 
  infinite-volume measure. At $q=4$, the random-cluster model should be 
  conformally invariant, but the probability of a crossing with free 
  boundary condition should converge to $0$. Nevertheless, the 
  probability that there is an open circuit surrounding the box of size 
  $n$ in the box of size $2n$ with free boundary condition should stay 
  bounded away from $0$.

  Proving an equivalent of Theorem~\ref{RSW_torus} with uniform 
  estimates with respect to boundary conditions is an important 
  question, since it would allow us to study the critical phase. The special 
  case $q=2$ has been derived recently in \cite{DuminilHonglerNolin}.
\end{remark}

\section{A sharp threshold theorem for crossing probabilities}
\label{sec:sharp_threshold}

The aim of this section is to understand the behavior of the function 
$p\mapsto \phi_{p,q,n}^\xi(A)$ for a non-trivial increasing event $A$. 
This increasing function is equal to $0$ at $p=0$ and to $1$ at $p=1$, 
and we are interested in the range of $p$ for which its value is between 
$\ep$ and $1-\ep$ for some positive $\ep$ (this range is usually 
referred to as a \emph{window}). Under mild conditions on $A$, the 
window will be narrow for large graphs, and its width can be bounded 
above in terms of the size of the underlying graph, which is known as a 
\emph{sharp threshold} behavior.

Historically, the general theory of sharp thresholds was first developed 
by Kahn, Kalai and Linial~\cite{KahnKalaiLinial} (see also 
\cite{Friedgut,FriedgutKalai,KalaiSafra}) in the case of product 
measures. In lattice models such as percolation, these results are used 
via a differential equality known as Russo's formula, 
see~\cite{GrimmettPerco,Russo}. Both sharp threshold theory and Russo's 
formula were later extended to random-cluster measures with $q\geq1$, 
see references below. These arguments being not totally standard, we 
remind the readers of the classic results we will employ and refer them 
to \cite{Grimmett} for general results. Except for 
Theorem~\ref{maximum_influence}, the proofs are quite short so that it 
is natural to include them. The proofs are directly extracted from the 
Grimmett's monograph \cite{Grimmett}.

\bigskip

In the whole section, $G$ will denote a finite graph; if $e$ is an edge 
of $G$, let $J_e$ be the random variable equal to $1$ if the edge $e$ is 
open, and $0$ otherwise. We start with an example of a differential 
inequality, which will be useful in the proof of 
Theorem~\ref{exponential_decay}.

\begin{proposition}[see \cite{Grimmett,GrimmettPiza}]
  \label{russo_hamming_proposition}
  Let $q\geq 1$; for any random-cluster measure $\phi_{p,q,G}^\xi$ with 
  $p\in(0,1)$ and any increasing event $A$, $$\frac{d}{dp} 
  \phi_{p,q,G}^\xi(A)\geq 4\phi_{p,q,G}^\xi(A)\phi_{p,q,G}^\xi(H_A),$$ 
  where $H_A(\omega)$ is the Hamming distance between $\omega$ and $A$.
\end{proposition}

\begin{proof}
  Let $A$ be an increasing event. The key step is the following 
  inequality, see \cite{BezuidenhoutGrimmettKesten,Grimmett}, which can 
  be obtained by differentiating with respect to $p$ (for details of the 
  computation, see Theorem (2.46) of \cite{Grimmett}):
  \begin{equation}\label{russo}
    \frac{d}{dp}\phi_{p,q,G}^\xi(A)=\frac{1}{p(1-p)}\sum_{e\in E} \left[ 
    \phi_{p,q,G}^\xi (\mathbbm{1}_AJ_e) - \phi_{p,q,G}^\xi (J_e) 
    \phi_{p,q,G}^\xi(A) \right].
  \end{equation}
  A similar differential formula is actually true for any random 
  variable $X$, but we will not use this fact in the proof. Define 
  $|\eta|$ to be the number of open edges in the configuration, it is 
  simply the sum of the random variables $J_e$, $e\in E$. With this 
  notation, one can rewrite \eqref{russo} as
  \begin{align*}
    \frac{d}{dp} \phi_{p,q,G}^\xi(A) &= \frac{1}{p(1-p)} \left[ 
    \phi_{p,q,G}^\xi (|\eta|\mathbbm{1}_A) - \phi_{p,q,G}^\xi (|\eta|) 
    \phi_{p,q,G}^\xi(A) \right] \\
    &= \frac{1}{p(1-p)} \left[ \phi_{p,q,G}^\xi \big((|\eta|+H_A) 
    \mathbbm{1}_A\big) - \phi_{p,q,G}^\xi \big(|\eta|+H_A\big) 
    \phi_{p,q,G}^\xi(A) \right. \\
    &\hspace{6em} \left. {} - \phi_{p,q,G}^\xi(H_A\mathbbm{1}_A) + 
    \phi_{p,q,G}^\xi (H_A) \phi_{p,q,G}^\xi(A) \right]\\
    &\geq \frac{1}{p(1-p)}\phi_{p,q,G}^\xi(H_A)\phi_{p,q,G}^\xi(A).
  \end{align*}
  To obtain the second line, we simply add and subtract the same 
  quantity. In order to go from the second line to the third, we remark 
  two things: in the second line, the third term equals $0$ (when $A$ 
  occurs, the Hamming distance to $A$ is $0$), and the sum of the first 
  two terms is positive thanks to the FKG inequality (indeed, it is easy 
  to check that $|\eta|+H_A$ is increasing). The claim follows since 
  $p(1-p)\leq 1/4$.
\end{proof}

This proposition has an interesting reformulation: integrating the 
formula between $p_1$ and $p_2>p_1$, we obtain
\begin{equation}\label{russo_hamming}
  \phi_{p_1,q,G}^\xi (A) \leq \phi_{p_2,q,G}^\xi (A) \; {\rm 
  e}^{-4(p_2-p_1)\phi_{p_2,q,G}^\xi(H_A)}
\end{equation}
(note that $H_A$ is a decreasing random variable).  If one can prove 
that the typical value of $H_A$ is sufficiently large, for instance 
because $A$ occurs with small probability, then one can obtain bounds 
for the probability of $A$. This kind of differential formula is very 
useful in order to prove the existence of a sharp threshold.  The next 
example presents a sharper estimate of the derivative.

Intuitively, the derivative of $\phi_{p,q,G}^\xi(A)$ with respect to $p$ 
is governed by the influence of one single edge, switching from closed 
to open (roughly speaking, considering the increasing coupling between 
$p$ and $p+{\rm d}p$, it is unlikely that two edges switch their state).  
The following definition is therefore natural in this setting. The 
\emph{(conditional) influence} on $A$ of the edge $e\in E$, denoted by 
$I_A(e)$, is defined as \[ I_A(e) := \phi_{p,q,G}^\xi (A|J_e=1) - 
\phi_{p,q,G}^\xi (A|J_e=0). \]

\begin{proposition}\label{russo_influence}
  Let $q\geq 1$ and $\ep>0$; there exists $c=c(q,\ep)>0$ such that for any 
  random-cluster measure $\phi_{p,q,G}^\xi$ with $p\in[\ep,1-\ep]$ and 
  any increasing event $A$, \begin{equation*}
    \frac{d}{dp}\phi_{p,q,G}^\xi(A)\geq c\sum_{e\in E}I_A(e).
  \end{equation*}
\end{proposition}

\begin{proof}
  Note that, by definition of $I_A(e)$, \[ \phi_{p,q,G}^\xi 
  (\mathbbm{1}_A J_e) - \phi_{p,q,G}^\xi (A) \phi_{p,q,G}^\xi (J_e) = 
  I_A(e) \phi_{p,q,G}^\xi (J_e) \big( 1 - \phi_{p,q,G}^\xi (J_e) \big) 
  \] so that \eqref{russo} becomes
  \begin{align*}
    \frac{d}{dp} \phi_{p,q,G}^\xi (A) &= \frac{1}{p(1-p)} \sum_{e\in E} 
    \phi_{p,q,G}^\xi (J_e) \big( 1 - \phi_{p,q,G}^\xi(J_e) \big) 
    I_A(e)\\
    &=\sum_{e\in E} \frac 
    {\phi_{p,q,G}^\xi(J_e)\big(1-\phi_{p,q,G}^\xi(J_e)\big)} {p(1-p)} 
    I_A(e)
  \end{align*}
  from which the claim follows since the term \[ \frac 
  {\phi_{p,q,G}^\xi(J_e)\big(1-\phi_{p,q,G}^\xi(J_e)\big)} {p(1-p)} \]
  is bounded away from $0$ uniformly in $p\in[\ep,1-\ep]$ and $e\in E$ 
  when $q$ is fixed.
\end{proof}

\bigskip

There has been an extensive study of the largest influence in the case 
of product measures. It was initiated in \cite{KahnKalaiLinial} and 
recently lead to important consequences in statistical models, see 
\emph{e.g.}~\cite{BollobasRiordan, BollobasRiordanperco}. The following 
theorem is a special case of the generalization to positively-correlated 
measures.

\begin{theorem}[see \cite{GrahamGrimmett}]\label{maximum_influence}
  Let $q\geq 1$ and $\ep>0$; there exists a constant $c = c(q,\ep) \in 
  (0,\infty)$ such that the following holds. Consider a random-cluster 
  model on a graph $G$ with $|E|$ denoting the number of edges of $G$.  
  For every $p\in[\ep,1-\ep]$ and every
  increasing event $A$, there exists $e\in E$ such that \[ I_A(e)\geq c 
  \,\phi_{p,q,G}^\xi(A)\big(1-\phi_{p,q,G}^\xi(A)\big)\frac{\log 
  |E|}{|E|}.  \]
\end{theorem}

There is a particularly efficient way of using 
Proposition~\ref{russo_influence} together with 
Theorem~\ref{maximum_influence}. In the case of a translation-invariant 
event on a torus of size $n$, horizontal (resp.\ vertical) edges play a 
symmetric role, so that the influence is the same for all of them. In 
particular, Proposition~\ref{russo_influence} together with 
Theorem~\ref{maximum_influence} provide us with the following 
differential inequality:
\begin{theorem}\label{symmetric_sharp_threshold}
  Let $q\geq1$ and $\ep>0$. There exists a constant $c=c(q,\ep) 
  \in(0,\infty)$ such that the following holds. Let $n\geq 1$ and let 
  $A$ be a translation-invariant event on the torus of size $n$: for any 
  $p\in[\ep,1-\ep]$, \[ \frac{d}{dp}\phi_{p,q,n}^{\rm p}(A)\geq c 
  \big(\phi_{p,q,n}^{\rm p}(A)(1-\phi_{p,q,n}^{\rm p}(A)\big)\log n.  \]
\end{theorem}

In particular, for a non-empty increasing event $A$, we can integrate
the previous inequality between two parameters $p_1<p_2$  (we recognize 
the derivative of $\log (x/(1-x))$) to obtain \[ 
\frac{1-\phi_{p_1,q,n}^{\rm p}(A)}{\phi_{p_1,q,n}^{\rm p}(A)}\geq 
\frac{1-\phi_{p_2,q,n}^{\rm p}(A)}{\phi_{p_2,q,n}^{\rm 
p}(A)}n^{c(p_2-p_1)}.\] If we further assume that 
$\phi_{p_1,q,n}^\xi(A)$ stays bounded away from $0$ uniformly in $n\geq 
1$, we can find $c'>0$ such that
\begin{equation}\label{important}
  \phi_{p_2,q,n}^{\rm p}(A)\geq 1-c'n^{-c(p_2-p_1)}.
\end{equation}
This inequality will be instrumental in the next section.

\section{The proofs of Theorems~\ref{critical_value} 
and~\ref{exponential_decay}}
\label{sec:proofs}

The previous two sections combine in order to provide estimates on 
crossing probabilities (see \cite{BollobasRiordan, BollobasRiordanperco} 
for applications in the case of percolation). Indeed, one can consider 
the event that \emph{some} long rectangle is crossed in a torus. At 
$p=p_{sd}$, we know that the probability of this event is bounded away 
from $0$ uniformly in the size of the torus (thanks to 
Theorem~\ref{RSW_torus}).  Therefore, we can apply 
Theorem~\ref{symmetric_sharp_threshold} to conclude that the probability 
goes to $1$ when $p>p_{sd}$ (we also have an explicit estimate on the 
probability). It is then an easy step to deduce a lower bound for the 
probability of crossing a \emph{particular} rectangle.

Theorem~\ref{critical_value} is proved by constructing a path from $0$ 
to infinity when $p>p_{sd}$, which is usually done by combining 
crossings of rectangles. There is a major difficulty in doing such a 
construction: one needs to transform estimates in the torus into 
estimates in the whole plane. One solution is to replace the periodic
boundary condition by wired boundary condition. The path construction is 
a little tricky since it must propagate wired boundary conditions 
through the construction (see Proposition~\ref{path_circuits}); it does 
not follow the standard lines.

Theorem~\ref{exponential_decay} follows from a refinement of the 
previous construction in order to estimate the Hamming distance of a 
typical configuration to the event $\{0\leftrightarrow 
\mathbb{L}\setminus [-n,n)^2\}$. It allows the use of 
Proposition~\ref{russo_hamming_proposition}, which improves bounds on 
the probability that the origin is connected to distance $n$. With these 
estimates, we show that the cluster size at the origin has finite 
moments of any order, whenever $p<p_{sd}$. Then, it is a standard step 
to obtain exponential decay in the sub-critical phase.

\bigskip

The following two lemmas will be useful in the proofs of both theorems.  
We start by proving that crossings of long rectangles exist with very 
high probability when $p>p_{sd}$.

\begin{lemma}\label{long_path}
  Let $\alpha>1$, $q\geq 1$ and $p>p_{sd}$; there exists 
  $\ep_0=\ep_0(p,q,\al)>0$ and $c_0=c_0(p,q,\al)>0$ such that
  \begin{equation}
    \phi_{p,q,\alpha^2 n}^{\rm p}\left(\C^v\big([0,n)\times[0,\alpha 
    n)\big)\right)\geq 1-c_0n^{-\ep_0}
  \end{equation}
  for every $n\geq1$.
\end{lemma}

\begin{proof}
  The proof will make it clear that it is sufficient to treat the case 
  of integer $\al$, we therefore assume that $\al$ is a positive integer 
  (not equal to $1$).  Let $B$ be the event that there exists a vertical 
  crossing of a rectangle with dimensions $(n/2,\al^2 n)$ in the torus 
  of size $\al^2n$. This event is invariant under translations and 
  satisfies
  \begin{equation*}
    \phi_{p_{sd},q,\al^2 n}^{\rm p}(B)\geq \phi_{p_{sd},q,\alpha^2 
    n}^{\rm p}\big(\C^v([0,n/2)\times[0,\al^2 n))\big)\geq c(2\al^2)
  \end{equation*}
  uniformly in $n$.

  Let $p>p_{sd}$. Since $B$ is increasing, we can apply 
  Theorem~\ref{symmetric_sharp_threshold} (more precisely 
  \eqref{important}) to deduce that there exist $\ep=\ep(p,q,\al)$ and 
  $c=c(p,q,\al)$ such that
  \begin{equation}\label{aa}
    \phi_{p,q,\alpha^2 n}^{\rm p}(B)\geq 1-cn^{-\ep}.
  \end{equation}
  If $B$ holds, one of the $2\alpha^3$ rectangles
  \begin{equation*}
    [in/2,in/2+n)\times[j\alpha n,(j+1)\alpha n), 
    \quad(i,j)\in\{0,\cdots,2\al^2-1\}\times\{0,\cdots,\al-1\}
  \end{equation*}
  must be crossed from top to bottom. We denote these events by $A_{ij}$ 
  --- they are translates of $\C^v([0,n)\times[0,\alpha n))$.  Using the 
  FKG inequality in the second line (this is another instance of the 
  ``square-root trick'' mentioned earlier), we find
  \begin{align*}
    \phi_{p,q,\alpha^2 n}^{\rm p}(B)&\leq 1-\phi_{p,q,\alpha^2 n}^{\rm 
    p}(B^c)\leq1-\phi_{p,q,\alpha^2 n}^{\rm p}(\cap_{i,j}A_{ij}^c)\\
    &\leq 1-\prod_{i,j}\phi_{p,q,\alpha^2 n}^{\rm 
    p}(A_{ij}^c)=1-\left[1-\phi_{p,q,\alpha^2 n}^{\rm 
    p}\big(\C^v([0,n)\times[0,\alpha n)\big)\right]^{2\alpha^3}.
  \end{align*}
  Plugging \eqref{aa} into the previous inequality, we deduce
  \begin{equation*}
    \phi_{p,q,\alpha^2 n}^{\rm p}\left(\C^v\big([0,n)\times[0,\alpha 
    n)\big)\right)\geq 1-(cn^{-\ep})^{1/(2\alpha^3)}.
  \end{equation*}
  The claim follows by setting $c_0:=c^{1/(2\alpha)^3}$ and 
  $\ep_0:=\ep/(2\alpha^3)$.
\end{proof}

\bigskip

Let $\al>1$ and $n\geq 1$; we define the annulus \[ A_n^\al := 
[-\al^{n+1},\al^{n+1}]^2\setminus [-\al^n,\al^n]^2.  \] An \emph{open 
circuit} in an annulus is an open path which surrounds the origin.  
Denote by $\A^\al_n$ the event that there exists an open circuit 
surrounding the origin and contained in $A_n^\al$, together with an open 
path from this circuit to the boundary of $[-\al^{n+2},\al^{n+2}]^2$, 
see Figure~\ref{fig:event_A_n}. The following lemma shows that the 
probability of $\A_n^\al$ goes to $1$, provided that $p>p_{sd}$ and that 
we fixed wired boundary conditions on $[-\al^{n+2},\al^{n+2}]^2$.

\begin{figure}[ht!]
  \begin{center}
    \includegraphics[width=0.7\hsize]{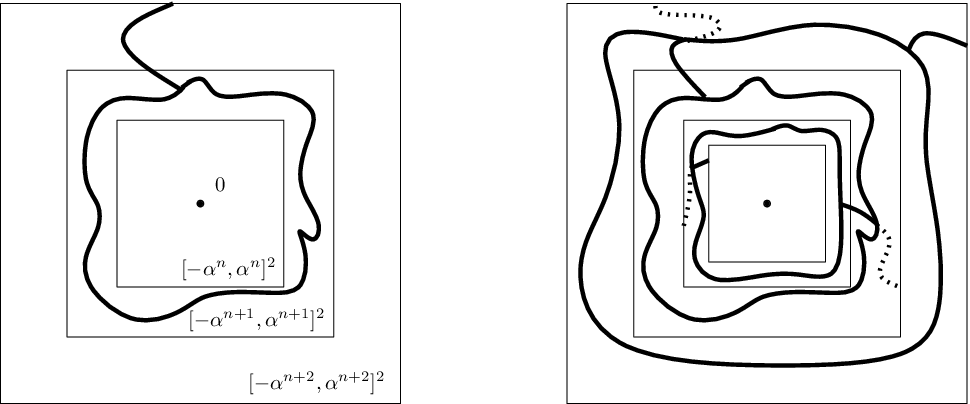}
  \end{center}
  \caption{\textbf{Left:} The event $\A_n^\alpha$.  \textbf{Right:} The 
  combination of events $\A_n^\alpha$: we see that it indeed constructs 
  a path from the origin to infinity.}
  \label{fig:event_A_n}
\end{figure}

\begin{lemma}\label{circuit}
  Let $\al>1$, $q\geq 1$ and $p>p_{sd}$; there exists $c_1=c_1(p,q,\al)$ 
  and $\ep_1=\ep_1(p,q,\al)$ such that for every $n\geq 1$,
  \begin{equation*}
    \phi_{p,q,\al^{n+2}}^{1}(\A^{\al}_n)\geq 1-c_1e^{-\ep_1n}.
  \end{equation*}
\end{lemma}

\begin{proof}
  First, observe that $\A^{\alpha}_n$ occurs whenever the following 
  events occur simultaneously:
  \begin{itemize}
    \item The following rectangles are crossed vertically: 
      \begin{align*}
        R_1 &:= 
        [\al^n,\alpha^{n+1}]\times[-\alpha^{n+1},\alpha^{n+1}],\\
        R_2 &:= 
        [-\alpha^{n+1},-\al^n]\times[-\alpha^{n+1},\alpha^{n+1}];
      \end{align*}
    \item The following rectangles are crossed horizontally: 
      \begin{align*}
        R_3 &:= 
        [-\alpha^{n+1},\alpha^{n+1}]\times[\al^n,\alpha^{n+1}],\\
        R_4 &:= 
        [-\alpha^{n+1},\alpha^{n+1}]\times[-\alpha^{n+1},-\al^n],\\
        R_5 &:= 
        [-\alpha^{n+2},\alpha^{n+2}]\times[-\alpha^{n+1},\alpha^{n+1}].
      \end{align*}
  \end{itemize}
  For the measure in the torus, these events have probability greater 
  than $1-c(\al^n)^{-\ep}$ with $c=c_0(p,q,2\al/(\al-1))$ and 
  $\ep=\ep_0(p,q,2\al/(\al-1))$. Using the FKG inequality, we obtain
  $$\phi_{p,q,\alpha^{n+2}}^{\rm p}(\A^\alpha_n)\geq 
  (1-c(\al^n)^{-\ep})^5$$ from which we deduce the following estimate, 
  harnessing the comparison between boundary conditions, 
  $$\phi_{p,q,\alpha^{n+2}}^{1} (\A^\alpha_n) \geq 
  (1-c(\al^n)^{-\ep})^5.$$ The claim follows by setting $c_1:=5c$ and 
  $\ep_1:=\ep\log \al$.
\end{proof}

\bigskip

The following proposition readily implies Theorem~\ref{critical_value}; 
It will also be useful in the proof of Theorem~\ref{exponential_decay}.  
We want to prove that the probability of the intersection of events 
$\A_n^\al$ is of positive probability when $p>p_{sd}$. So far, we know 
that there is an open circuit with very high probability when we 
consider the random-cluster measure with wired boundary condition in a 
slightly larger box. In order to prove the result, we assume the 
existence of a large circuit. Then, we iteratively condition on events 
$\A_{n-k}^\al$, $k\geq 0$. When conditioning `from the outside to the 
inside', we guarantee that at step $k$, there exists an open circuit in 
$A_{n-k+1}^\al$ that surrounds $A_{n-k}^\al$. Using comparison between 
boundary conditions, we can assure that the measure in $A_{n-k}^\al$ 
stochastically dominates the measure in $A_{n-k+1}^\al$ with wired 
boundary condition. In other words, we keep track of advantageous 
boundary conditions. Note that the reasoning, while reminiscent of 
Kesten's construction of an infinite path for percolation, is not 
standard.  

\begin{proposition}\label{path_circuits}
  Let $\alpha>1$, $q\geq 1$ and $p>p_{sd}$; there exist 
  $c,c_1,\varepsilon_1>0$ (depending on $p$, $q$ and $\alpha$) such that 
  for every $N\geq 1$, $$\phi_{p,q}\left(\bigcap_{n\geq 
  N}\A_N^{\alpha}\right)\geq c\prod_{k=N}^{\infty}(1-c_1{\rm 
  e}^{-\ep_1k}) > 0.$$
\end{proposition}

\begin{proof}
  Let $\al>1$, $q\geq1$, $p>p_{sd}$, $N \geq 1$ and recall that there is 
  a unique infinite-volume measure $\ph$. For every $n\geq 1$, we know 
  that
  \begin{equation}\label{b}
    \ph\left(\bigcap_{k=N}^n\A_n^\al\right)= 
    \ph(\A_n^\al)\prod_{k=N}^{n-1}\ph(\A^\al_k|\A^\al_j,k+1\leq j\leq 
    n).
  \end{equation}

  On the one hand, let $k\in[N,n-1]$. Conditionally to $\A^\al_j$, 
  $k+1\leq j\leq n$, we know that there exists a circuit in the annulus 
  $A_{k+1}^\al$.  Exploring from the outside, we shall consider the most 
  exterior such circuit, denoted by $\Ga$. Conditionally to $\Ga=\g$, 
  the unexplored part of the box $[-\al^{k+2},\al^{k+2}]^2$ follows the 
  law of a random-cluster configuration with wired boundary condition.  
  In particular, the conditional probability that there exists a circuit 
  in $A^\al_k$ connected to $\g$ is greater than the probability that 
  there exists a circuit in $A^\al_k$ connected to the boundary of 
  $[-\al^{k+2},\al^{k+2}]^2$ with wired boundary condition. Therefore, 
  we obtain that almost surely
  \begin{align*} \ph(\A^\al_{k}|\A^\al_{j},k+1\leq j\leq 
    n)&=\ph\big(\ph(\A^\al_k|\Ga=\g)\big)\\
    &\geq\ph\big(\phi_{p,q,\al^{k+2}}^{1}(\A^\al_k)\big)\\
    &\geq 1-c_1{\rm e}^{-\ep_1k}
  \end{align*}
  where we have harnessed Lemma~\ref{circuit} in the last inequality. 

  On the other hand, for $p=p_{sd}$, consider the event $\A_n^{\al}$ in 
  the bulk. Thanks to Corollary~\ref{RSW_bulk}, its probability is 
  bounded away from $0$ uniformly in $n$. Since the event is increasing, 
  we obtain that there exists $c=c(\al)>0$ such that
  \begin{equation*}
    \ph(\A_n^\al)=\ph^1(\A_n^\al)\geq c
  \end{equation*}
  for any $n\geq N$ and $p> p_{sd}$. Plugging the two estimates into 
  \eqref{b}, we obtain \[ \ph\left(\bigcap_{k=N}^n\A_n^\al\right)\geq 
  c\prod_{k=N}^{n-1}(1-c_1{\rm e}^{-\ep_1k})\geq 
  c\prod_{k=N}^{\infty}(1-c_1{\rm e}^{-\ep_1k}).  \] Letting $n$ go to 
  infinity concludes the proof.
\end{proof}

\bigskip

\begin{proof}[Proof of Theorem~\ref{critical_value}]
  The bound $p_c\geq p_{sd}$ is provided by Zhang's argument, as 
  explained in Section~\ref{sec:basic}. For $p>p_{sd}$, fix $\alpha>1$.  
  Applying Proposition~\ref{path_circuits} with $N=1$, we find \[ 
  \ph(0\leftrightarrow\infty)\geq c\ph\left( \bigcap_{n\geq1}\A_n^\al 
  \right)>0 \] so that $p$ is super-critical. The constant $c>0$ is due 
  to the fact that we require $[-\alpha^2,\alpha^2]^2$ to contain open 
  edges only ($c>0$ exists using the finite energy property). Since $p$ 
  is super-critical for every $p>p_{sd}$, we deduce $p_c\leq p_{sd}$.
\end{proof}

\bigskip

\begin{proof}[Proof of Theorem~\ref{exponential_decay}]
  Let $x$ be a site of $\Z^2$, and let $\C_x$ be the cluster of $x$, 
  \ie\  the maximal connected component containing the site $x$. We 
  denote by $|\C_x|$ its cardinality. We first prove that $|\C_x|$ has 
  finite moments of any order. Then we deduce that the probability of 
  $\{|\C_x|\geq n\}$ decays exponentially fast in $n$.  The proof of the 
  Step 2 is extracted from \cite{Grimmett}.

  \paragraph{Step 1: finite moments for $|\C_x|$.}

  Let $d>0$ and $p<p_{sd}$; we want to prove that
  \begin{equation}\label{moment_estimate}
    \ph(|\C_x|^d)<\infty.
  \end{equation}
  In order to do so, let $p_1 := (p+p_{sd})/2$ and define $D_n := 
  \{x\leftrightarrow \Z^2\setminus (x+[-n,n)^2)\}$; denote by $H_n$ the 
  Hamming distance to $D_n$. Note that $H_n$ is the minimal number of 
  closed edges that one must cross in order to go from $x$ to the 
  boundary of the box of size $n$ centered at $x$. Let $$\al := \exp 
  \left[ \frac{p_1-p}{2d+3} \right]>1.$$ We know from 
  Proposition~\ref{path_circuits}, applied to the (super-critical) dual 
  model, that the probability of $\bigcap_{n>N}(\A_n^\al)^\star$ is 
  larger than $c\prod_N^\infty (1-c_1e^{-\varepsilon_1n}) > 0$ 
  ($(\A_n^\al)^\star$ is the occurrence of $\A_n^\al$ in the dual 
  model). Hence, there exists $N=N(p_1,q,\al)$ sufficiently large such 
  that \[ \phi_{p_1,q}\left(\bigcap_{k\geq N}^\infty (\A_n^\al)^\star 
  \right)\geq \frac12. \]On this event, $H_{n}$ is greater than $(\log 
  n/\log \al)-N$ since there is at least one closed circuit in each 
  annulus $A_k^\al$ with $k\geq N$ (thus increasing the Hamming distance 
  by 1).  We obtain \[ \phi_{p_1,q}(H_{n})\geq \left(\frac{\log n}{\log 
  \al}-N\right)\phi_{p_1,q}\left(\bigcap_{k \geq N}^\infty 
  (\A_n^\al)^\star \right)\geq\frac{\log n}{4\log \al} \] for $n$ 
  sufficiently large. We can use \eqref{russo_hamming} to find 
  \begin{equation}\label{exp}
    \phi_{p,q}(D_n)\leq \phi_{p_1,q}(D_n)\exp 
    \big[-4(p_1-p)\phi_{p_1,q}(H_n)\big]\leq n^{-(2d+3)}
  \end{equation}
  for $n$ sufficiently large, from which \eqref{moment_estimate} follows 
  readily.

  \paragraph{Step 2: exponential decay.}

  Note that, from the first inequality of \eqref{exp}, it is sufficient 
  to prove that for some constant $c>0$, $$\liminf_{n\rightarrow \infty} 
  H_n/n \geq c\quad \text{a.s.}$$ in order to show that 
  $\phi_{p,q}(D_n)$ decays exponentially fast.

  Consider a (not necessarily open) self-avoiding path $\g$ going from 
  the origin to the boundary of the box of size $n$. We can bound from 
  below the number $T(\g)$ of closed edges along this path by the 
  following quantity: \[ \frac{T(\g)}{n}\geq\frac1{|\g|}T(\g) \geq 
  \frac1{|\g|}\sum_{z\in \g} \frac{1}{|\C_z|}\geq 
  \left(\frac1{|\g|}\sum_{z\in \g}|\C_z|\right)^{-1}.  \] Indeed, the 
  number of closed edges in $\g$ is larger than the number of distinct 
  clusters intersecting $\g$.  Moreover, if $\mathcal{C}$ denotes such a 
  cluster, we have that $1\geq \sum_{z\in\g}|\C|^{-1}\mathbbm{1}_{z\in 
  \C}$. The last inequality is due to Jensen's inequality. Since $H_n$ 
  can be rewritten as the infimum of $T(\gamma)$ on paths going from $0$ 
  to the boundary of the box, we obtain
  \begin{equation}\label{expression_H_n}
    \frac{H_n}n\geq \inf_{\g:0\leftrightarrow 
    \Z^2\setminus\B_n}\left(\frac{1}{|\g|}\sum_{z\in 
    \g}|\C_z|\right)^{-1}.
  \end{equation}
  The goal of the end of the proof is to give an almost sure lower bound 
  of the right-hand side. We will harness a two-dimensional analogue of 
  the strong law of large number. In order to do that, we need to 
  transform the random variables $|\C_z|$ to obtain independent 
  variables. We start with the following domination.

  Let $(\tilde{\C}_z)_{z\in\B_n}$ be a family of independent subsets of 
  $\Z^2$ distributed as $\C_z$. We claim that $(|{\C}_z|)_{z\in \B_n}$ 
  is stochastically dominated by the family $(M_z)_{z\in \B_n}$ defined 
  as $$M_z:=\sup_{y\in \Z^2:z\in\tilde{\C}_y}|\tilde{\C}_y|.$$

  Let $v_1, v_2, \ldots$ be a deterministic ordering of $\mathbb{Z}^2$.  
  Given the random family $(\tilde{\C}_z)_{z\in\B_n}$, we shall 
  construct a family $(D_z)_{z\in \B_n}$ having the same joint law as 
  $(\C_z)_{z\in \B_n}$ and satisfying the following condition: for each 
  $z$, there exists $y$ such that $D_z\subset\tilde{\C}_y$. First, set 
  $D_{v_1}=\tilde{C}_{v_1}$.  Given $D_{v_1}$, $D_{v_2}$, \ldots, 
  $D_{v_n}$, define $E=\bigcup_{i=1}^n D_{v_1}$. If $v_{n+1}\in E$, set 
  $D_{v_{n+1}}=D_{v_j}$ for some $j$ such that $v_{n+1}\in D_{v_j}$. If 
  $v_{n+1}\notin E$, we proceed as follows. Let $\Delta_eE$ be the set 
  of edges of $\Z^2$ having exactly one end-vertex in $E$. We may find a 
  (random) subset $F$ of $\tilde{C}_{v_{n+1}}$ such that $F$ has the 
  conditional law of $C_{n+1}$ given that all edges in $\Delta_eE$ are 
  closed; we now set $D_{v_{n+1}}=F$. We used the domain Markov property 
  and the positive association. Indeed, we use that the law of 
  $C_{v_{n+1}}$ depends only on $\Delta_eE$, and is stochastically 
  dominated by the law of the cluster in the bulk without any 
  conditioning.  We obtain the required stochastic domination 
  accordingly. In particular, $|\C_z|\leq M_z$ and $M_z$ has finite 
  moments.

  From \eqref{expression_H_n} and the previous stochastic domination, we 
  get
  \begin{equation*}
    \liminf_{n\rightarrow \infty}\frac{H_n}n\geq 
    \liminf_{n\rightarrow\infty}\inf_{\g:0\leftrightarrow 
    \Z^2\setminus\B_n}\left(\frac{1}{|\g|}\sum_{z\in 
    \g}|\C_z|\right)^{-1}\geq 
    \left(\limsup_{n\rightarrow\infty}\sup_{\g:0\leftrightarrow 
    \Z^2\setminus\B_n}\frac{1}{|\g|}\sum_{z\in \g}M_z\right)^{-1}.
  \end{equation*}
  The second step is now to replace $M_z$ by random variables that are 
  independent. We can harness Lemma~2 of \cite{FontesNewman} to show 
  that
  \begin{equation*}
    \left(\limsup_{n\rightarrow\infty}\sup_{\g:0\leftrightarrow 
    \Z^2\setminus\B_n}\frac{1}{|\g|}\sum_{z\in \g}M_z\right)^{-1}\geq 
    \left(2\limsup_{n\rightarrow\infty}\sup_{|\Ga|\geq n}
    \frac{1}{|\Ga|}\sum_{z\in \g}|\tilde{\C}_z|^2\right)^{-1}
  \end{equation*}
  where the supremum is over all finite connected graphs $\Gamma$ of 
  cardinality larger than $n$ that contain the origin (also called 
  lattice animals). 

  Since the $|\tilde{\C}_z|^2$ are independent and have finite moments 
  of any order, the main result of \cite{CoxGandolfiGriffinKesten} 
  guarantees that

  $$2\limsup_{n\rightarrow\infty}\sup_{|\Ga|\geq n} 
  \frac{1}{|\Ga|}\sum_{z\in \g}|\tilde{\C}_z|^2\leq C\quad a.s.$$ for 
  some $C>0$. Therefore, with positive probability, $\liminf H_n/n$ is 
  greater than a given constant, which concludes the proof.
\end{proof}

\section{The critical point for the triangular and hexagonal lattices}
\label{sec:triangular}

Let $\mathbb T$ be the triangular lattice of mesh size $1$, embedded in 
the plane in such a way that the origin is a vertex and the edges of 
$\mathbb{T}$ are parallel to the lines of equations $y=0$, 
$y=\sqrt{3}x/2$ and $y=-\sqrt{3}x/2$. The dual graph of this lattice is 
a hexagonal lattice, denoted by $\mathbb{H}$, see 
Figure~\ref{fig:reseau_triangular}. Via planar duality, it is sufficient 
to handle the case of the triangular lattice in order to prove 
Theorem~\ref{triangular}. Define $p_{\mathbb{T}}$ as being the unique 
$p\in (0,1)$ such that $y^3+3y^2-q=0$, where $y := p_{\mathbb{T}} / 
(1-p_{\mathbb{T}})$. The goal is to prove that 
$p_c(\mathbb{T})=p_{\mathbb{T}}$. 

\begin{figure}[ht!]
  \begin{center}
    \includegraphics[width=0.8\hsize]{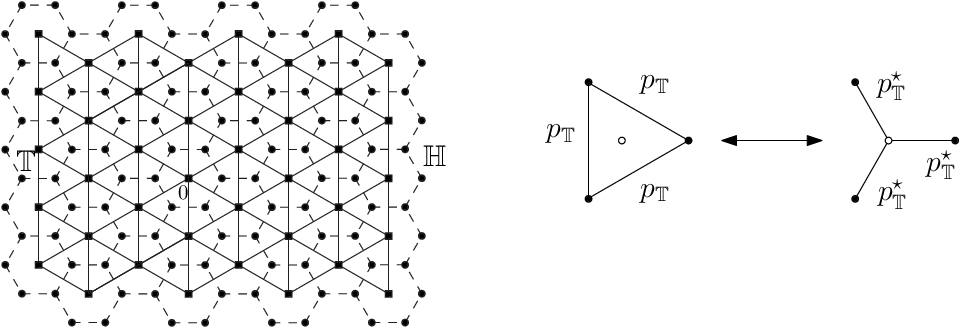}
  \end{center}
  \caption{\textbf{Left:} The triangular lattice $\mathbb{T}$ with its 
  dual lattice $\mathbb{H}$.    \textbf{Right:} The exchange of the two 
  patterns does not alter the random-cluster connective properties of 
  the black vertices.}
  \label{fig:reseau_triangular}
\end{figure}

\bigskip

The general strategy is the same as in the square lattice case. We prove 
that at $p=p_\mathbb{T}$, a crossing estimate similar to 
Theorem~\ref{RSW_torus} holds. Sharp threshold arguments and proofs of 
Section~\ref{sec:proofs} can be adapted \emph{mutatis mutandis}, 
replacing square-shaped annuli by hexagonal-shaped annuli. The crossing 
estimate must be slightly modified, and we present the few changes. It 
harnesses the planar-duality between the triangular and the hexagonal 
lattices, and the so-called \emph{star-triangle transformation} (see 
\emph{e.g.} Section 6.6 of \cite{Grimmett} and 
Figure~\ref{fig:reseau_triangular}). We assume that the reader is 
already familiar with the star-triangle transformation.

Let $e_1=\sqrt{3}/2+{\rm i}/2$ and $e_2={\rm i}$; whenever we write 
coordinates, they are understood as referring to the basis $(e_1,e_2)$.  
A `rectangle' $[a,b)\times[c,d)$ is the set of points in $z\in 
\mathbb{T}$ such that $z=\lambda e_1+\mu e_2$ with $\lambda\in[a,b)$ and 
$\mu\in[c,d)$ (it has a lozenge shape, see \emph{e.g.} 
Figure~\ref{fig:event_A_triangular}). By analogy with the case of the 
square lattice, $\C_v(D)$ denotes the event that there exists a path 
between the top and the bottom sides of $D$ which stays inside $D$. Such 
a path is called a \emph{vertical open crossing} of the rectangle. Other 
quantities are defined similarly. Let $\mathbb{T}_m$ be the torus of size $m$ constructed using the 'rectangle' of the form $[0,m]\times[0,m]$ with respect to the basis $(e_1,e_2)$. We present the crossing estimate in 
the case of the torus $\mathbb{T}_m$ (deriving the bulk 
estimate follows the same lines as in the square lattice case); 
$\phi_{p_{sd},q,m}^{\rm p}$ denotes the random-cluster measure on 
$\mathbb{T}_m$.

\begin{theorem}\label{RSW_torus_triangular}
  Let $\al>1$ and $q\geq 1$. There exists $c(\al)>0$ such that for every 
  $m>\alpha n>0$,
  \begin{equation}
    \phi_{p_\mathbb{T},q,m}^{\rm p}\big(\C_h([0,n)\times[0,\al 
    n))\big)\geq c(\al).
  \end{equation}
\end{theorem}

The main difficulty is the adaptation of Lemma~\ref{self-duality}.  
Define the line $d:=-\sqrt 3/3+{\rm i}\R$. The orthogonal symmetry 
$\sigma_d$ with respect to $d$ maps $\mathbb{T}$ to another triangular 
lattice. Note that this lattice is a sub-lattice of $\mathbb{H}$ (in the 
sense that its vertices are also vertices of $\mathbb H$).  Let $\g_1$ 
and $\g_2$ be two paths satisfying the following Hypothesis $(\star)$, 
see Figure~\ref{fig:triangular_symmetry}:
\begin{itemize}
  \item $\g_1$ remains on the left of $d$ and $\g_2$ remains on the 
    right,
  \item $\g_2$ begins at $0$ and $\g_1$ begins on a site of 
    $\mathbb{T}\cap (-\sqrt 3/ 2+{\rm i}\mathbb{R}_+)$,

  \item $\g_1$ and $\sigma_d(\g_2)$ do not intersect (as curves in the 
    plane),
  \item $\g_1$ and $\sigma_d(\g_2)$ end at two sites (one primal and one 
    dual) which are at distance $\sqrt3/3$ from one another.
\end{itemize}
When following the paths in counter-clockwise order, we can create a 
circuit by linking the end points of $\g_1$ and $\sigma_d(\g_2)$ by a 
straight line, the start points of $\sigma_d(\g_2)$ and $\g_2$, the end 
points of $\g_2$ and $\sigma_d(\g_1)$, and the start points of 
$\sigma_d(\g_1)$ and $\g_1$. The circuit $(\g_1,\sigma_d(\g_2), \g_2, 
\sigma_d(\g_1))$ surrounds a set of vertices of $\mathbb{T}$. Define the 
graph $G(\g_1,\g_2)$ with sites being site of $\mathbb{T}$ that are 
surrounded by the circuit $(\g_1,\sigma_d(\g_2),\g_2,\sigma_d(\g_1))$, 
and with edges of $\mathbb{T}$ that remain entirely inside the circuit 
(boundary included). 

\begin{figure}[ht!]
  \begin{center}
    \includegraphics[width=0.4\hsize]{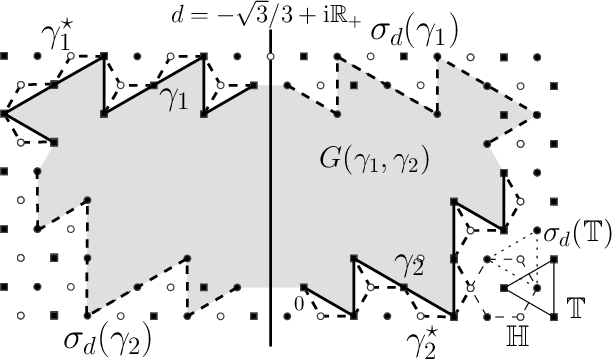}
  \end{center}
  \caption{The graph $G(\g_1,\g_2)$ with the two solid arcs $\g_1$ and 
  $\g_2$ and the dashed arcs $\sigma_d(\g_1)$ and $\sigma_d(\g_2)$. The 
  dual arcs $\g_1^\star$ and $\g_2^\star$ are 
  dotted.}\label{fig:triangular_symmetry}
\end{figure}

We will need an additional technical condition, which we present now.  
Note that for any edge of $\sigma_d(\mathbb{T})$ there is one vertex of 
$\mathbb{T}$ and one vertex of $\mathbb{H}$ at distance $\sqrt 3/6$ from 
its midpoint. We assume that for any edge of $\sigma_d(\g_1)$ and 
$\sigma_d(\g_2)$, the associated vertex of $\mathbb{T}$ is in the 
interior of the domain $G(\g_1,\g_2)$ (therefore, the associated vertex 
of $\mathbb{H}$ is outside the domain, see white vertices in 
Fig~\ref{fig:triangular_symmetry}). We will refer to this condition as 
Hypothesis $(\star\star)$.

The \emph{mixed boundary condition} on this graph is wired on $\g_1$ 
(all the edges are pairwise connected), wired on $\g_2$, and free 
elsewhere.  We denote the measure on $G(\g_1,\g_2)$ with parameters 
$(p_\mathbb{T},q)$ and mixed boundary condition by 
$\phi_{p_\mathbb{T},q,\g_1,\g_2}$ or more simply $\phi_{\g_1,\g_2}$.  
With these definitions, we have an equivalent of 
Lemma~\ref{self-duality}:

\begin{lemma}\label{self-duality_triangular}
  For any $\g_1,\g_2$ satisfying Hypotheses $(\star)$ and 
  $(\star\star)$, we have $$\phi_{\g_1,\g_2}(\g_1\leftrightarrow 
  \g_2)\geq \frac 1{1+q^2}.$$
\end{lemma}

\begin{proof}
  As previously, if $\g_1$ and $\g_2$ are not connected, $\g_1^\star$ 
  and $\g_2^\star$ are connected in the dual model, where 
  $\g_1^\star,\g_2^\star\subset\mathbb{H}$ are the dual arcs bordering 
  $G(\g_1,\g_2)$ close to $\sigma_d(\g_1)$ and $\sigma_d(\g_2)$. Thanks 
  to Hypothesis $(\star\star)$ and the mixed boundary condition, this 
  event is equivalent to the event that $\sigma_d(\g_1)$ and 
  $\sigma_d(\g_2)$ are dual connected. Using Hypothesis $(\star)$ and 
  the symmetry, we deduce $$\phi_{\g_1,\g_2} \big(\g_1\leftrightarrow 
  \g_2\big) + \sigma_d*\phi_{\g_1,\g_2}^\star 
  \big(\g_1\leftrightarrow\g_2\big)=1,$$ where as before $\sigma_d * 
  \phi_{\g_1,\g_2}^\star$ denotes the push-forward under the symmetry 
  $\sigma_d$ of the dual measure of $\phi_{\g_1,\g_2}$ --- in 
  particular, it lies on $\sigma_d(\mathbb{H})$ and the edge-weight is 
  $p_\mathbb{T}^\star$.  This lattice contains the sites of $\mathbb{T}$ 
  and those of another copy of the triangular lattice which we will 
  denote by $\mathbb{T}'$.  Since $\g_1$ and $\g_2$ are two paths of 
  $\mathbb{T}$, one can use the star-triangle transformation for any 
  triangle of $\mathbb{T}$ included in $G(\g_1,\g_2)$ that contains a 
  vertex of $\mathbb{T}'$: one obtains that 
  $\sigma_d*\phi_{\g_1,\g_2}^\star\big(\g_1\leftrightarrow\g_2\big)$ is 
  equal to the probability of $\g_1$ and $\g_2$ being connected, in a 
  model on $\mathbb{T}$ with edge-weight $p_\mathbb{T}$. Here, we need 
  Hypothesis $(\star\star)$ again in order to ensure that all the 
  triangles containing a vertex of $\mathbb{T}'$ have no edges on the 
  boundary (which would have forbidden the use of the star-triangle 
  transformation). The same observation as in the case of the square 
  lattice shows that the boundary conditions are the same as for 
  $\phi_{\g_1,\g_2}$, except that arcs $\g_1$ and $\g_2$ are wired 
  together. The same reasoning as in Lemma~\ref{self-duality} implies 
  that $$\sigma_d*\phi_{\g_1,\g_2}^\star \big( \g_1\leftrightarrow\g_2 
  \big) \leq q^2 \phi_{\g_1,\g_2}\big(\g_1\leftrightarrow \g_2\big),$$ 
  and the claim follows readily.
\end{proof}

The existence of $c(1)$ is obtained in the same way as in the case of 
the square lattice, with only the obvious modifications needed; we leave 
the details as an ``exercise for the reader''.  
Theorem~\ref{RSW_torus_triangular} is derived exactly as in 
Section~\ref{sec:crossings}, as soon as an equivalent of 
Proposition~\ref{box_crossing} holds:
\begin{proposition}\label{box_crossing_triangular}
  There exists a constant $c(3/2)>0$ such that, for all $m>3n/2>0$, 
  $$\phi_{p_{\mathbb{T}},q,m}^{\rm p} \big( \C^v([0,3n/2)\times[0,n)) 
  \big) \geq c(3/2).$$
\end{proposition}

\begin{figure}[ht!]
  \begin{center}
    \includegraphics[width=0.8\hsize]{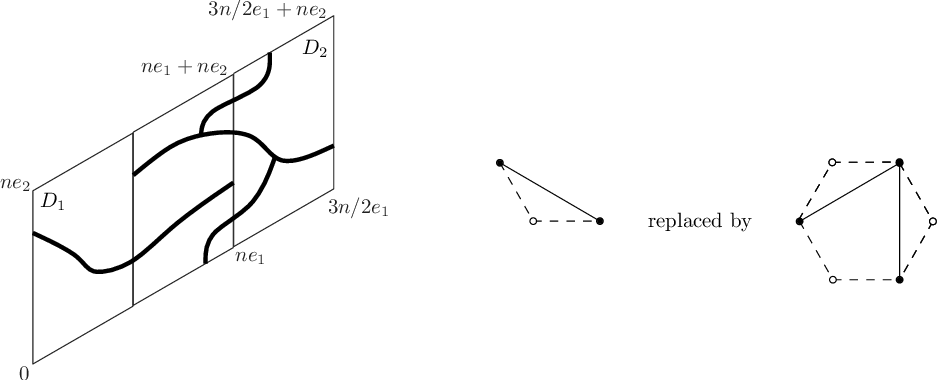}
  \end{center}
  \caption{\textbf{Left:} The set $[0,3n/2)\times[0,n)$ and the event 
  $A$. \textbf{Right:} One can obtain the path $\Ga_1'$ from $\Ga_1$ by 
  replacing any bad edge with two edges. Since $\Ga_1$ is the top-most 
  crossing, it contains no double edges and this construction can be 
  done.}
  \label{fig:event_A_triangular}
\end{figure}

\begin{proof}
  The general framework of the proof is the same as before, but some 
  technicalities occur because the underlying lattice is not self-dual.  
  Consider the rectangle $D=[0,3n/2)\times[0,n)$, which is the union of 
  rectangles $D_1=[0,n)\times[0,n)$ and $D_2=[n/2,3n/2)\times[0,n)$, see 
  Figure~\ref{fig:event_A_triangular}. Let $A$ be the event that:
  \begin{itemize}
    \item $D_1$ and $D_2$ are both crossed horizontally (each crossing 
      has probability at least $c(1)$ to occur);
    \item $[n/2,n)\times\{0\}$ (resp. $[n,3n/2)\times\{n\}$) is 
      connected inside $D_2$ to the top side (resp.\ to the bottom).  
      Using the FKG inequality and symmetries of the lattice, this event 
      occurs with probability larger than $c(1)^2/4$.
  \end{itemize}
  Therefore, $A$ has probability larger than $c(1)^4/4$. 

  When $A$ occurs, define $\Ga_1$ to be the top-most crossing of the 
  rectangle $D_1$, and $\Ga_2$ the right-most crossing in $D_2$ between 
  $[n/2,n)\times\{0\}$ and the top side of $D_2$. Note that $\Ga_2$ is 
  automatically connecting $[n/2,n)\times\{0\}$ to the right edge and to 
  $[n,3n/2)\times\{n\}$. In order to conclude, it is sufficient to prove 
  that $\Ga_1$ and $\Ga_2$ are connected with probability larger than 
  some positive constant.

  Consider the lowest path $\Ga'_1$ above $\Ga_1$ which satisfies the 
  following property: for any edge $e$ in $\Ga'_1$, the associated site 
  of $\sigma_d(\mathbb{H})$ (see the definition of Hypothesis 
  $(\star\star)$) is in the connected component of $D_1\setminus \Ga'_1$ 
  \emph{above} $\Ga_1'$. Such a path can be obtained from $\Ga_1$ by 
  replacing every `bad' edge with the other two edges of a triangle, as 
  shown in Figure~\ref{fig:event_A_triangular}. Since $\Ga_1$ is the 
  top-most crossing, it cannot have double edges and the path $\Ga'_1$ 
  can be constructed. In particular it ends at the same point as 
  $\Ga_1$, and it goes from left to right. Note that it is not 
  necessarily open. We define $\Ga_2'$ similarly in the obvious way (the 
  left-most path on the right of $\Ga_2$ such that for any edge of 
  $\Ga_2'$, the associate site of $\sigma_d(\mathbb{H})$ is on the right 
  of $\Ga'_2$).

  We now sketch the end of the proof. Apply a construction similar to 
  the proof of Proposition~\ref{box_crossing} in order to create a 
  domain $G(\Ga'_1,\Ga_2')$. With mixed boundary conditions, the 
  probability of connecting $\Ga'_1$ to $\Ga'_2$ in $G(\Ga'_1,\Ga'_2)$ 
  is larger than $1/(1+q^2)$ ($\Ga_1'$ and $\Ga_2'$ have been 
  constructed in such a way that Hypothesis $(\star\star)$ is
  fulfilled). But $\Ga_1$ disconnects $\Ga'_1$ from $\Ga'_2$, and 
  $\Ga_2$ disconnects $\Ga'_2$ from $\Ga_1$. Using boundary conditions 
  inherited from the fact that $\Ga_1$ and $\Ga_2$ are crossings, one 
  can prove that $\Ga_1$ is connected to $\Ga_2$ in $G(\Ga'_1,\Ga_2')$ 
  with probability larger than $1/(1+q^2)$. The end of the proof follows 
  exactly the same lines as in the case of the square lattice.
\end{proof}

\paragraph{Acknowledgments.}

The authors were supported by the ANR grant BLAN06-3-134462, the EU 
Marie-Curie RTN CODY, the ERC AG CONFRA, as well as by the Swiss {FNS}.  
The authors would like to thank Geoffrey Grimmett for many fruitful 
discussions and precious advice, and Andr\'as B\'alint for numerous 
comments on the manuscript. This work was mostly done during a stay of 
the authors at IMPA in Rio de Janeiro: they would like to thank Vladas 
Sidoravicius for his hospitality. The second author would like to thank 
Stanislav Smirnov for his constant support.

\begin{flushright}\footnotesize\obeylines
  \textsc{Unit\'e de Math\'ematiques Pures et Appliqu\'ees}
  \textsc{\'Ecole Normale Sup\'erieure de Lyon}
  \textsc{F-69364 Lyon CEDEX 7, France}
  \textsc{E-mail:} \texttt{Vincent.Beffara@ens-lyon.fr}
  \bigskip
  \textsc{Section de Math\'ematiques}
  \textsc{Universit\'e de Gen\`eve}
  \textsc{Gen\`eve, Switzerland}
  \textsc{E-mail:} \texttt{hugo.duminil@unige.ch}
\end{flushright}
\end{document}